\documentclass[11pt,twoside,reqno]{amsart}

\usepackage{amssymb}
\usepackage{graphicx}

\def\const{\text{\rm const}}
\def\spec{\text{\rm spec}}

\def\ti{\tilde}

\def\ker{\text{\rm ker}}

\def\PW{\text{\rm PW}}
\def\dist{\text{\rm dist}}
\def\supp{\text{\rm supp}\,}
\def\to{\rightarrow}

\def\sign{\text{\rm sign}}

\def\ms{\medskip}
\def\no{\noindent}

\def\tsim{\scalebox{.7}{$\ \overset{\scriptscriptstyle{\bf T}}{\sim}\ $}}

\def\tleq{\scalebox{.7}{$\ \overset{\scriptscriptstyle{\bf T}}{\leqslant}\ $}}
\def\tgeq{\scalebox{.7}{$\ \overset{\scriptscriptstyle{\bf T}}{\geqslant}\ $}}
\def\tl{\scalebox{.7}{$\ \overset{\scriptscriptstyle{\bf T}}{<}\ $}}

\def\seq{\scalebox{.7}{$\ \overset{\boldsymbol\cdot}{=}\ $}}

\def\R{{\mathbb R}}
\def\T{{\mathbb T}}

\def\D{{\mathbb{D}}}
\def\Z{{\mathbb{Z}}}
\def\C{{\mathbb{C}}}
\def\N{{\mathbb{N}}}

\def\NN{{\mathcal{N}}}

\def\ZZ{{\mathcal{Z}}}

\def\DD{{\mathcal{D}}}

\def\b{{\textbf b}}
\def\MM{{\mathcal M}}
\def\CC{{{\textsl{C}}}}

\def\SS{{\mathcal S}}

\def\TT{{\mathcal T}}

\def\e{\varepsilon}

\def\L{\Lambda}
\def\l{\lambda}
\def\G{\Gamma}
\def\lan{\lambda_n}

\theoremstyle{plain}
\newtheorem{lemma}{Lemma}
\newtheorem{theorem}{Theorem}
\newtheorem{corollary}{Corollary}
\newtheorem{proposition}{Proposition}
\newtheorem{claim}{Claim}

\newtheorem{definition}{Definition}
\newtheorem{example}{Example}

\numberwithin{equation}{section}

\author{A.~Poltoratski}
\address{Texas A\&M University
\\ Department of Mathematics\\
College Station, TX 77843, USA\\
and \\
Department of Mathematics and Mechanics\\ St. Petersburg State University\\ St. Petersburg, Russia }
\email{alexeip@math.tamu.edu}
\thanks{Initial  work on this article (Sections 1-5) was supported by
NSF Grant DMS-1665264. The article was completed (Sections 6-9) with support of RNF grant 14-41-00010}

\title{Toeplitz Order}

\begin{document}

\begin{abstract} A new approach to problems of the Uncertainty Principle in Harmonic Analysis, based on the use of Toeplitz operators, has brought progress to some of the classical problems in the area. The goal of this paper is to develop and systematize the function theoretic component of the Toeplitz approach by introducing a partial order on the set of inner functions induced by the action of Toeplitz operators. We study connections of the new order
with some of the classical problems and known results. We discuss remaining problems and possible directions for further research.
\end{abstract}

\maketitle

\tableofcontents

\ms\section{Introduction}

\ms\no Toeplitz operator $T_U$ with symbol $U\in L^\infty(\R)$ on the Hardy space $H^2$ in the upper half-plane $\C_+$ is defined as
$$T_U f=P_+Uf,$$
where $P_+$ denotes the orthogonal projection from $L^2(\R)$ onto $H^2$ (see Section \ref{secT} for further discussion).
This standard definition can be extended to larger function spaces and more general symbols to accommodate various applications of Toeplitz-type operators in Complex and Harmonic analysis.
A recently developed
approach based on the use of Toeplitz operators brought new progress to the area of Uncertainty Principle in Harmonic Analysis (UP), see
for instance \cite{MIF1, MIF2, CBMS}. This note is devoted to further development of the Toeplitz approach.

\ms\no One of the cases of the Toeplitz operator which appears most often in applications is the operator with the symbol $U=\bar IJ$ where $I$ and $J$ are
inner functions. Recall that a bounded analytic function in the upper half-plane is called inner if its boundary
values are unimodular almost everywhere with respect to Lebesgue measure on the boundary.

\ms\no Inner functions constitute arguably the most important collection of functions in the standard one-dimensional complex function theory. Starting with the seminal
result by Beurling, which says that all closed invariant subspaces of the shift operator $Sf:f\mapsto zf$ in the Hardy space $H^2$ in the unit disk have the form
$\theta H^2$ where $\theta$ is inner, these functions became a focal point of research for complex analysts. Beurling's result implies that the invariant
subspaces for the operator adjoint to $S$, the backward shift operator $S^*f:f\mapsto (f-f(0))/z$, have the form $K_\theta=(\theta H^2)^\perp=H^2\ominus \theta H^2$.
This property of the spaces $K_\theta$ put them into the foundation of the famous Nagy-Foias functional model theory which says that
any  completely non-unitary contractive operator $T$ in a Hilbert space $H$, satisfying $||(T^*)^nx||\to 0$ for all $x\in H$, is unitarily equivalent to a compression of multiplication by $z$ on one
of such $K_\theta$ spaces for a properly chosen inner function $\theta$ (in general, such spaces are vector-valued, see \cite{Nikolski}).

\ms\no These fundamental results demonstrated the importance of inner functions and related spaces in function theoretic problems stemming from functional analysis. Such problems
became the main stream of complex function theory in the last several decades of the 20th century. At present, inner functions are firmly established as a key ingredient of complex analysis and appear in most of its applications, including Harmonic Analysis, Control Theory, Spectral Theory of differential operators, Signal Processing and Mathematical Physics. Via the same connections, Toeplitz operators of the type $T_{\bar IJ}$ where $I$ and $J$ are inner functions, appear in many of such applications.


\ms\no Problems on injectivity and invertibility of Toeplitz operators with symbols $\bar I J$  have been known to play
crucial role in the study of Riesz bases, frames and completeness in various function spaces, see for instance \cite{HNP, MIF1}. As was mentioned before, recently
such operators have become a central object in the  Toeplitz approach to UP  \cite{MIF1, MIF2, CBMS}. Via the Toeplitz approach, these and similar operators apply to many fields of analysis including questions in Fourier analysis and spectral problems for differential operators, see for instance \cite{MIF1, CBMS, Uncertainty, Etudes}.

\ms\no Intuitively, the property that the Toeplitz operator $T_{\bar I J}$ has a non-trivial kernel means that $I$ is, in some sense, larger than $J$. Similarly,
invertibility of such an operator indicates that $I$ and $J$ are 'equivalent' or have roughly the same 'size'. However, as we will
discuss in Section \ref{secEx}, neither of these properties
 can yield a formal definition of  order or equivalence, since they lack axiomatic properties of transitivity and reflexivity correspondingly.

 \ms\no In this note we attempt to fix this problem and 'lift' these intuitive notions to the level of formal order and equivalence. Via the Toeplitz approach
 the new order encompasses a variety of problems and applications mentioned above.
It reveals relations between  problems of Complex and Harmonic
analysis and helps to systematize some of the well-known questions from the area of UP and its applications.
 The goal of this paper is to present the basic definitions and properties of  Toeplitz order, outline its connections
 with known problems, and to suggest further directions for research.

\ms\no {\bf Acknowledgment}: This paper is based on a mini-course given at the Universite Paul Sabatier, Toulouse, in October of 2016, as a part of
CIMI Thematic Semester in Analysis. I am very grateful to the organizers
Serban Belinschi, Stefanie Petermichl and Pascal Thomas  for giving me a reason and an opportunity to collect my thoughts on this subject.

\section{Preliminaries}

\subsection{Inner functions}
In this note we will mostly concern ourselves with inner functions in the upper half-plane $\C_+$.
Such functions can be represented as
a product
$$I=B_\L J_\mu,$$
where $B_\L$ is the Blaschke product corresponding to the sequence $\L=\{\lan\}\subset\C_+$ of zeros of $I$ and $J_\mu$ is a singular inner
function corresponding to a positive  singular measure $\mu$ on $\hat\R=\R\cup\{\infty\}$. The measure  can be represented as $\mu=\nu+c\delta_\infty$ where $\nu$ is Poisson-finite on $\R$, i.e.,
$$\int \frac{d\nu(x)}{1+x^2}<\infty,$$
 and $c\geq 0$ is the mass at infinity. The singular function $J_\mu$ is defined as
$$J_\mu=e^{-\SS\mu}=e^{-\SS\nu+icz},$$
where $\SS\mu$ is the Schwarz integral of $\mu$:
$$\SS \mu(z)=\frac1{\pi i} \int\left[\frac1{t-z}-\frac t{1+t^2}\right]d\mu(t).$$

\ms\no The Blaschke Product $B_\L$   for $\L=\{\lan\}$ is defined as
$$B_\L=\prod c_n\frac{z-\lan}{z-\bar\lan},$$
where $c_n$ are unimodular constants chosen so that $c_n\frac{i-\lan}{i-\bar\lan}>0$. If $\L$ is an infinite sequence
then the necessary and sufficient condition for the normal convergence of the partial products of $B_\L$ is
that $\L$ satisfies the Blaschke condition
$$\sum \frac {\Im \lan}{1+|\lan |^2}<\infty.$$

\ms\no Throughout this paper we will use the notation $S^a(z)=e^{iaz}$ for the complex exponential function, which is the singular inner function corresponding to the pointmass $a>0$ at
infinity. Using our notations $S^a=J_{a\delta_\infty}$.

\ms\no Similar statements and formulas are true for the case of the unit disk, see for instance \cite{Garnett, KoosisHp}.

\ms\no   A special role in our notes will be played by meromorphic inner functions (MIF) which are inner functions in the upper half-plane
that can be extended meromorphically to the whole plane. The above formulas imply that an inner function is a MIF if and only if
its Blaschke factor corresponds to a discrete sequence $\L\subset \C$ (a sequence without finite accumulation points) and
the measure in the singular factor is a point mass at infinity, i.e. $J_\mu=S^a=e^{iaz}$ for some non-negative $a$.

\subsection{Model spaces and Clark theory}\label{secMod}

 For each inner function $\theta(z)$ one may consider a model subspace
$$K_\theta=H^2\ominus \theta H^2$$ of the Hardy space
$H^2=H^2(\C_+)$. Here '$\ominus$' stands for the orthogonal difference, i.e., $K_\theta$ is the orthogonal
complement of the space $\theta H^2=\{ \theta f| f\in H^2\}$ in $H^2$. As was mentioned in the introduction, these subspaces play an important role in complex and
harmonic analysis, as well as in operator theory,
see~\cite{Nikolski}.

\ms\no   Each inner function $\theta(z)$ defines a positive harmonic
function
$$\Re \frac{1+\theta(z)}{1-\theta(z)}$$
 and, by the Herglotz
representation, a positive measure $\sigma$ such that
\begin{equation} \label{for1} \Re
\frac{1+\theta(z)}{1-\theta(z)}=py+\frac{1}{\pi}\int{\frac{yd\sigma
(t)}{(x-t)^2+y^2}}, \hspace{1cm} z=x+iy,\end{equation} for some $p
\geq 0$. The number $p$ can be viewed as a point mass at infinity.
The measure $\sigma$ is  a singular Poisson-finite measure, supported on the set where non-tangential limits
of $\theta$ are equal to $1$.
The measure
$\sigma +p\delta_\infty$ on $\hat\R=\R\cup \{\infty\}$
is called the Clark measure for $\theta(z)$.

\ms\no Following  standard notations, we will sometimes denote the Clark measure defined in \eqref{for1} by $\sigma_1$.
If $\alpha\in \C, |\alpha|=1$, then $\sigma_\alpha$ is the measure defined by \eqref{for1}
with $\theta$ replaced by $\bar\alpha\theta$.
In some settings it is convenient to call the measue $\sigma_{-1}$ the 'Clark dual' of the measure $\sigma_1$.

\ms\no   Conversely, for
every positive Poisson-finite singular measure $\sigma$  and a number $p \geq
0$, there exists an inner function $\theta(z)$ satisfying~\eqref{for1}.

\ms\no   Every function $f \in K_\theta$ has non-tangential boundary values
$\sigma_1$-a.e. and can be recovered from these values via the formula
\begin{equation} \label{for2} f(z)=\frac{p}{2\pi
i}(1-\theta(z))\int{f(t)\overline{(1-\theta(t))}dt}+\frac{1-\theta(z)}{2\pi
i}\int{\frac{f(t)}{t-z} d\sigma (t)}
\end{equation}
see \cite{PoltClark}.
If the Clark measure does not have a point mass at infinity,
the formula is simplified to
\begin{equation}f(z)=(1-\theta(z))Kf\sigma\label{for2a}\end{equation}
where $Kf\sigma$ stands for the Cauchy integral
\begin{equation}Kf\sigma(z)=\frac1{2\pi i}\int\frac{f(t)}{t-z} d\sigma (t).\label{CauchyIntegral}\end{equation}
This gives an isometry of
$L^2(\sigma)$ onto $K_\theta$.
The Clark measure $\sigma_1$ has a point mass at infinity if and only if $1-\theta(t) \in
L^2(\R)$.

\ms\no Similar formulas can be written for any $\sigma_\alpha$ corresponding to $\theta$.
For any $\alpha,\ |\alpha|=1$ and any $f\in K_\theta$, $f$ has non-tangential boundary
values $\sigma_\alpha$-a.e. on $\hat\R$. Those boundary values can be used in \eqref{for2} or \eqref{for2a} to
recover $f$.

 \ms\no  In the case of meromorphic $\theta(z)$ (MIF),
every function $f \in K_\theta$ also has a meromorphic extension in
$\C$, and it is given by the formula~\eqref{for2}.


\ms\no   Each meromorphic inner function $\theta(z)$ can be written as
$\theta(t)=e^{i\phi(t)}$ on $\R$, where $\phi(t)$ is a real analytic
and strictly increasing function. The function $\phi(t)=\arg{\theta(t)}$ is
a continuous branch of the argument of $\theta(z)$.

 \ms\no   For any inner function $\theta$ in the upper half-plane we define its spectrum $\spec_\theta$ as the closure of the set
$\{\theta=1\}$, the set of points on the line where the non-tangential limit of $\theta$ is equal to $1$,
plus the infinite point if the corresponding Clark measure has a point mass at infinity, i.e. if
$p$ in \eqref{for1} is positive. If $\spec_\theta\subset\R$, then
$p$ in \eqref{for1} is $0$.


 \ms\no  Recall that a sequence of real points is discrete if it has no finite accumulation points. Note
that $\spec_\theta$ is discrete if and only if $\theta$ is  meromorphic. The
corresponding Clark measure is discrete with masses at the points
of the set $\{\theta=1\}$ given by
$$\sigma(\{x\})=\frac{2\pi
}{|\theta'(x)|},$$
plus possibly a point mass at infinity (related similarly to the derivative at infinity).

\ms\no   If $\L\subset \R\ (\hat\R)$ is a given discrete sequence, one can easily construct a meromorphic
inner function $\theta$ satisfying $\{\theta=1\}=\L$ by  considering a positive
Poisson-finite measure concentrated on $\L$ and then choosing $\theta$ to satisfy
\eqref{for1}. One can prescribe the derivatives of $\theta$ at $\L$ with a proper choice
of pointmasses.

\ms\no The same construction shows that an arbitrary continuous growing function $\gamma$ on $\R$ can
be approximated, up to a bounded function, by the argument of a meromorphic inner function.
If $\L=\{\gamma=2\pi n\}$ then $\theta$, constructed as above with $\{\theta=1\}=\L$, satisfies
$|\gamma-\arg \theta|<2\pi$ on $\R$. Furthermore, if $\G=\{\gamma=(2n+1)\pi\}$ one can easily construct
$\theta$ so that $\{\theta=1\}=\L$ and $\{\theta=-1\}=\G$ and achieve an even better approximation
$|\gamma-\arg \theta|<\pi$.


\ms\no  For more information and further references on Clark measures
see \cite{PolSar}.

\subsection{Toeplitz kernels}\label{TKSection}\label{secT}

  Recall that the Toeplitz operator $T_U$ with a symbol $U\in
L^\infty(\R)$ is the map
$$T_U:H^2\to H^2,\qquad F\mapsto P_+(UF),$$
where $P_+$ is the Riesz projection, i.e. the orthogonal projection  from $L^2(\R)$ onto the
Hardy space $H^2$. Passing from a function in $H^2$ to its  non-tangential boundary values on $\R$,
$H^2$ can be identified with a closed subspace of $L^2(\R)$  formed by functions $f\in L^2(\R)$ whose Fourier transform $\hat f$ is supported on $[0,\infty)$,
which makes the Riesz projection
correctly defined.

\ms\no   We will use the following notation for kernels of Toeplitz
operators (or {\it Toeplitz kernels}) in $H^2$:
$$N[U]=\ker{\ T_U}.$$
An important observation is that $N[\bar\theta]=K_\theta$ if $\theta$ is an inner
function. Along with $H^2$-kernels, one may consider Toeplitz kernels $N^p[U]$
in other Hardy classes $H^p$, the kernel $N^{1,\infty}[U]$ in the 'weak' space $H^{1,\infty}=H^p\cap L^{1,\infty},\ 0<p<1$, or the kernel in the Smirnov class $\mathcal{N}^+(\C_+)$, defined as $$ N^+[U]=\{f \in \mathcal{N}^+ \cap
L^1_{loc}(\R): \bar{U}\bar{f} \in \mathcal{N}^+\}$$
for $\mathcal{N}^+$ and similarly for other spaces.

\ms\no   If $\theta$ is a meromorphic inner function, $K^+_\theta = N^+[\bar\theta]$ can also be considered. For more on such kernels see  \cite{CBMS}.

\subsection{Entire functions and de Branges spaces}\label{dBrSpacesSection}\label{secE}

  Recall that an entire function $F(z)$ is said to be
of exponential type at most $a>0$
if
$$|F(z)|=O(e^{a|z|})$$
as $z\to\infty$. The infimum of such $a$ is the exponential type of $F$. Throughout the paper we denote by
$\Pi$ the Poisson measure on $\R,\ d\Pi(x)=dx/(1+x^2)$.

\ms\no   A classical theorem of Krein gives a connection between the
Smirnov class $\mathcal{N}^+(\C_+)$ and the Cartwright class
$\CC_a$ consisting of all entire functions $F(z)$ of exponential
type $\leq  a$ which satisfy $$\log|F(t)| \in L^1_\Pi.$$ An
entire function $F(z)$ belongs to the Cartwright class $\CC_a$ if
and only if
$$ \frac{F(z)}{S^{- a}(z)} \in N^+(\C_+), \ \ \textrm{ and }\ \
\frac{F^{\#}(z)}{S^{- a}(z)} \in N^+(\C_+),$$ where
$F^\#(z)=\overline{F(\bar{z})}$.

\ms\no Recall that a Paley-Wiener space $PW_a$ is defined as a space
of entire functions of exponential type at most $a$ which belong
to $L^2(\R)$. Equivalently, $PW_a=\CC_a\cap L^2(\R)$.
  As an immediate consequence one obtains a connection between the
Hardy space $H^2(\C_+)$ and the Paley-Wiener space $\PW_a$. Namely,
an entire function $F(z)$ belongs to the Paley-Wiener class $\PW_a$
if and only if
$$ \frac{F(z)}{S^{- a}(z)} \in H^2(\C_+), \hspace{1cm}
\frac{F^{\#}(z)}{S^{- a}(z)} \in H^2(\C_+).$$

\ms\no   The definition of the de~Branges spaces of entire functions may be
viewed as a generalization of the last definition of the
Paley-Wiener spaces with $S^{-a}(z)$ replaced by a more general
entire function. Consider an entire function  $E(z)$ satisfying
the inequality
$$ |E(z)|>|E(\bar{z})|, \ \  z \in \C_{+}.$$ Such
functions are usually called  de~Branges functions. The de~Branges
space $B(E)$ associated with $E(z)$ is defined to be the space of
entire functions $F(z)$ satisfying
$$ \frac{F(z)}{E(z)} \in H^2(\C_+), \hspace{1cm}
\frac{F^{\#}(z)}{E(z)} \in H^2(\C_+).$$ It is a Hilbert space
equipped with the norm $\|F\|_E =\|F/E\|_{L^2(\R)}.$ If $E(z)$ is of exponential type then
all the functions in the de~Branges space $B(E)$ are of
exponential type not greater then the type of $E(z)$  (see, for example, the last part in the proof of lemma 3.5 in~\cite{DM}). A de~Branges
space is called short (or regular) if together with
every function $F(z)$ it contains $(F(z)-F(a))/(z-a)$ for any
$a\in\C$.

\ms\no One of the most important features of de Branges spaces is that they admit a second, axiomatic, definition. Let $H$ be
a Hilbert space of entire functions that satisfies the following axioms:

\begin{itemize}

\item (A1) If $F\in H,\ F(\l)=0$, then $\frac{F(z)(z-\bar\l)}{z-\l}\in H$ with the same norm;

\item (A2) For any $\l\not\in\R$, point evaluation at $\l$ is a bounded linear functional on $H$;

\item (A3) If $F\in H$ then $F^\#\in H$ with the same norm.

\end{itemize}

\ms\no Then $H=B(E)$ for a suitable de Branges function $E$. This is theorem 23 in \cite{dBr}.

\ms\no Usually, for a given Hilbert space of entire functions it is not difficult to check the above axioms and conclude
that the space is a de Branges space (if the axioms do hold). It is however a challenging problem in many situations to find a generating
function $E$. This problem can be viewed as a deep and abstract generalization of the inverse spectral problem for
second order differential operators.

\ms\no Every de Branges space $B(E)$ is a reproducing kernel Hilbert space, i.e., for each point $\l\in\C$ there exists a function $k_\l\in B(E)$ such that
$$F(\l)=<F,k_\l>$$
for any $F\in B(E)$. The reproducing kernel $k_\lambda$ is given by the formula
$$k_\lambda(z)=\frac{E(z)\bar E(\lambda)-E^\#(z)E(\bar\lambda)}{2\pi i(\bar\lambda-z)}.
$$

\ms\no
It is not difficult to show that for any de Branges function $E$ sequences of reproducing kernels $\{k_\l\}_{\l\in\L}$, where $\L\subset\R,\ \L=\{E^\#/E=\alpha\}$
for some constant $\alpha, |\alpha|=1$, form orthogonal bases of $B(E)$. Moreover, these are the only orthogonal bases of reproducing kernels.

\ms\no De Branges spaces possess the so called nesting property, which makes Krein-de Branges theory
especially suitable to study spectral problems for differential operators. It says that for  any two de Branges spaces
$B(E_1)$ and $ B (E_2)$ isometrically embedded into a third de Branges space, either
$B(E_1)\subset B (E_2)$ or
$B(E_2)\subset B(E_1)$. It follows that any space $B(E)$ admits a unique chain of subspaces $B(E_t), \ 0\leq t\leq 1$ monotone by inclusion with
$E_0=\const$ and $E_1=E$ (in the case of so-called jump intervals the parameter $t$ may not take all values from 0 to 1). Moreover,
for any positive Poisson-finite measure $\mu$ on $\R$ there is a unique regular chain of de Branges spaces isometrically embedded into $L^2(\mu)$.

 \ms\no  Every de~Branges function $E(z)$ gives rise to a MIF $$\theta(z) =\theta_E(z)=E^\#(z)/E(z)$$ and a model space $K_{\theta}$ that this inner function generates. There exists a well known isometric isomorphism between $B(E)$ and $K_{\theta}$ given by $F\to F/E$. Conversely, for every MIF $\theta$
 there exists a de Branges function $E$ such that $\theta=\theta_E$. Such a function $E$ is unique up to a multiplication by a real entire function without zeros in $\C\setminus \R$ (an entire function is called real if it is real on $\R$).  We call a de Branges function $E$ an Hermite-Biehler (HB) function if it has no zeros on the real line.
 For a given MIF $\theta$
one can always choose the corresponding de Branges function $E$ to be an HB function.

 \ms\no
  As was mentioned above, if $\theta$ is a MIF then all  Clark measures $\sigma_{\alpha}$ of $\theta$ are discrete and their point masses can be computed by $\sigma_{\alpha}(\lambda)=2\pi/|{\theta'(\lambda)}|$ for $\lambda\in\{\theta=\alpha\}$.
We will call the measures $|E|^2\sigma_{\alpha}$, where $\sigma_\alpha$ is a Clark measure for $\theta(z)=E^\#(z)/E(z)$,  spectral measures of the corresponding de~Branges space. It is well known (and follows from a similar property for Clark measures) that for any spectral measure $\nu$ of a de Branges space $B(E)$ the natural embedding gives
an isometric isomorphism between $B(E)$ and $L^2(\nu)$. This isomorphism generalizes the Parseval theorem.

\ms\no   On the real line each inner $\theta(z)$ coming from a de~Branges function can be written as $\theta(t)=e^{i\phi(t)},\ t\in\R$, where $\phi(t)$ is real analytic strictly increasing function, a continuous branch of the argument of $\theta(z)$ on $\R$. The phase function of the corresponding de~Branges space is defined by $\psi(t)=\phi(t)/2$ and is equal to $-\arg E$.


\ms\no Throughout this paper we will utilize the following notations for the objects discussed in the last several sections.
If $E$ is an HB function we will denote by $\theta_E$ the corresponding MIF $\theta_E=E^\#/E$. If $\mu$ is a positive singular Poisson-finite
measure on $\hat \R$ we denote by $\theta_\mu$ the inner function with the Clark measure equal to $\mu$.
Even though for a MIF $\theta$ the function $E$ such that $\theta=\theta_E$ is not unique, we will use the notation $E_\theta$ for
one of such functions. The reader may think of a function with lowest order and type among all such HB functions $E$.

\section{Toeplitz Order (TO) and Equivalence (TE)}

\subsection{Main definitions} In this section we use Toeplitz operators to define partial order and equivalence on the
set of inner functions in the upper half-plane. Our definitions can be naturally extended to broader classes of functions
and measures (see Section \ref{secEx} below), however in this note we choose to concentrate on the inner case. Moreover, in most applications
discussed in the rest of this paper, the inner functions are meromorphic (MIFs).

\ms\no We begin with the following definition. Recall that $N[U]$ denotes the $H^2$-kernel of the Toeplitz operator with symbol $U$.

\begin{definition}
If $\theta $ is an inner function we define its (Toeplitz) dominance set $\DD(\theta)$ as
$$\DD(\theta)=\{ I\textrm{ inner }|\ N[\bar\theta I]\neq 0\}.$$
\end{definition}

\ms\no Every collection of  sets admits natural partial ordering by inclusion. In our case, we consider dominance sets $\DD(\theta)$
as subsets of the set of all inner functions in the upper half-plane and the partial order $\subset$ on this collection.
This partial order induces a preorder on the set of all inner functions in $\C_+$. Proceeding in a standard way, we can modify this preorder into a partial
order by introducing equivalence classes of inner functions. The details of this definition are as follows.

\begin{definition}
We will say that two inner functions $I$ and $J$ are Toeplitz equivalent, writing $I\tsim J$, if $\DD(I)=\DD(J)$. This equivalence relation divides
the set of all inner functions in $\C_+$ into equivalence classes. We call this relation \textbf{\textit{Toeplitz equivalence}} (TE).
\end{definition}

\ms\no Further, we introduce a partial order on these equivalence classes
defining it as follows.

\begin{definition} We write $I\tleq J$ (meaning that the equivalence class of $I$ is 'less or equal' than the equivalence class of $J$) if
$\DD(I)\subset \DD(J)$.
We call this  partial order on the set of inner functions in $\C_+$
\textit{\textbf{Toeplitz order}} (TO).
\end{definition}

\ms\no The following simple examples illustrate our definitions.

\begin{example}
Let $B_n$ and $B_k$ be  Blaschke products of degree $n$ and $k$ correspondingly. Then $B_n\tsim B_k$ iff $n=k$ and $B_n\tl B_k$ iff $n<k$.

\ms\no If $J_\mu$ and $J_\nu$ are two singular functions, $J_\mu\tleq J_\nu$ if $\nu-\mu$ is a non-negative measure. However, there exist
$\mu$ and $\nu$ such that $\mu\perp \nu$ but $J_\mu\tleq J_\nu$, as follows from an example given in \cite{BS}.
\end{example}

\ms\no It is a good exercise on Toeplitz kernels to establish the statements of the above example.

\subsection{Extensions of TO and other orderings}\label{secExt}\label{secEx}
As was explained in Section \ref{secMod}, Clark theory provides a natural one-to-one correspondence between inner functions in $\C_+$
and positive singular Poisson-finite measures on $\hat\R$. Via this correspondence one may introduce Toeplitz equivalence and order
on the set of all such measures. I.e., for any two positive singular Poisson-finite measures $\mu$ and $\nu$, $\mu\tsim\nu$ if $I_\mu\tsim I_\nu$
and $\mu\tleq\nu$ if $I_\mu\tleq I_\nu$. Similarly, Toeplitz order on inner functions induces an order on Hermite-Biehler functions, de Branges spaces,
canonical systems, model spaces, model contractions, etc.

\ms\no In the same way one can order the set of all unimodular functions on the real line. If $U=e^{i\phi}$ is a unimodular function on $\R$ (and $\phi$ is
a measurable real function) then one can defined its dominance set $\DD(U)$ as the set of inner functions $\theta$ such that $N[\bar U\theta]$ is non-trivial.
After that, once again using the ordering of dominance sets by inclusion, one can introduce equivalence classes on the set of unimodular functions
and partial order on those classes.
In a slightly different way, one may view the ordering described above as an ordering of equivalence classes of real measurable functions $\phi$ on $\R$ defined as $\phi\tleq\psi$ if $e^{i\phi}\tleq e^{i\psi}$. Analogously, TO can be moved from the upper half-plane to the unit disk or even more general domains. Without any changes in the above definitions, TO can be extended to bounded analytic functions in $\C_+$ or even unbounded functions if one is
willing to deal with unbounded Toeplitz operators.

\ms\no Using quadratic forms one can consider Toeplitz operators with distributional symbols. If $m$ is a distribution on $\R$ then
$\DD(m)$ can be defined as the set of inner functions such that $T_{\theta\bar m}$ exists and has a non-trivial kernel. After finding a way to
overcome obvious technical difficulties in this definition, one can proceed with an extension of TO to this class. In particular, one
obtains a different way to extend TO to the set of measures and it may be interesting to study relations with the extension outlined above.

\ms\no Perhaps the simplest way to order inner functions is by division, i.e., to say that $I\leq J$ if $I$ divides $J$ (if $J/I$ is an inner function).
The main flaw of the order by division is that most pairs of inner functions remain incomparable. It is easy to see that TO is an extension
of the order by division since $I\tleq J$ whenever $I$ divides $J$, see Section \ref{secdiv}. While for two functions to be comparable in the order by division the zero set
of one has to be a subset of the zero set of the other, in TO one zero set only needs to be 'near' the other.

\ms\no Another way to define an order on inner functions is to say that $I>J$ if $N^\infty[\bar IJ]\neq 0$ or if $N^+[\bar IJ]\neq 0$ (the kernel in the Smirnov
class $\NN^+$). These orders are different from ours. The $N^+$-order is related to (and used implicitly in) the Beurling-Malliavin theory. This order is meaningful, but less
relevant to  problems discussed in these notes. As follows from Lemma \ref{lemdom}, TO is a proper extension of the $H^\infty$-version of the above order.

\ms\no While all  versions of Toeplitz order mentioned in this section seem to be interesting, in this note we will concentrate on the inner version of TO in $\C_+$ as defined in the last section.

\section{Structure of the dominance set $\DD(\theta)$}\label{DS1Section}

\ms\no The goal of this section is to study the dominance set $\DD(\theta)$, the key element of Toeplitz order.
We will identify two important subsets of $\DD(\theta)$, the sets of base and total elements, and discuss their relations with adjacent questions.

\subsection{Difference of arguments}\label{secdif}

\ms\no Let $I,J$ be two MIFs and let us denote by $\phi=\phi(I,J)$ the difference of arguments $\frac12(\arg I-\arg J)$. Recall that
the argument of a MIF on the real line can be chosen as a real analytic function and therefore the last expression makes sense.
If $\phi(I,J)$ is Poisson-summable then its harmonic conjugate $\ti\phi$ exists and we will denote by $h(I,J)$
the outer function $\exp(\ti \phi -i\phi)$. Note that then $h(I,J)=1/h(J,I).$

\ms\no Clearly, a sufficient condition for $I\tsim J$ is that $\phi$ has a bounded harmonic conjugate, i.e.,
$$0<c<h<C<\infty$$
on $\R$ for some constants $c$ and $C$. Indeed,
in that case $f\in N[\bar I L]$ iff $hf \in N[\bar J L]$, which implies that $\DD(I)=\DD(J)$.
However, this condition is not necessary.

\begin{example} \label{exone} Let $I=B_\L$, where $\lan=2^{n}(1+i), n\in \N$. Notice that then $|I'|=O(1/x)$ on
$\R$ as $x\to\infty$.
Let us construct $J$ in the following way.
For a rare subsequence $n_k=2^k,k\in \N,$ pull the zeros closer to the real line, i.e., define
$$J=I\prod \left(d_k\frac{z-(2^{n_k}+ic_k)}{z-\l_{n_k}}\cdot\frac{z-\bar\l_{n_k}}{z-(2^{n_k}-ic_k)}\right),$$
where $c_k>0$ are positive constants tending to zero and $d_k$ are convergence constants. Then $\phi=\phi(J,I)$ satisfies $0<c<\exp(\ti \phi)$
on $\R$, because
$$\exp(\ti \phi)=\prod q_k \frac{1- z/{\bar\l_{n_k}}}{1- z/{(2^{n_k}-ic_k)}}$$
with proper convergence constants $q_k$. One can  show that if $c_k$ tend to zero slow enough (say, $c_k=1/k$) we also have
$$\exp(\ti \phi)\in L^2(\R, |I'|^{}dx)\setminus L^\infty(\R).$$
 If $f\in N[\bar I L]$ for some MIF $L$
then $hf\in N[\bar J L]$ because $f$ belongs to $K_I$ and hence is bounded on $\R$ by $C|I'|^{1/2}$.
Conversely, if $f\in N[\bar J L]$ then $f/h\in N[\bar I L]$ because $1/h$ is bounded. Therefore
$I\tsim J$ even though $\ti\phi$ is unbounded.
\end{example}

\subsection{Base and total elements of $\DD(\theta)$}\label{secB}

We say that an inner function $I\in \DD(\theta)$ is a base element if it does not  divide any other
element of $\DD(\theta)$. In other words, base
 elements are the maximal elements of $\DD(\theta)$ with respect
to the order by division. We will denote by $\DD_B(\theta)$ the set of all base elements of $\DD(\theta)$.

\ms\no We denote by $b_a$ the Blaschke factor with zero at $a\in\C_+$:
$$b_a=\frac {\bar a} a\frac{z-a}{z-\bar a}.$$
If $\theta(a)=0$ for some $a\in\C_+$ then $\theta_a=\theta/b_a$ is a base element of $\DD(\theta)$. More generally, one can show that
if $\theta(c)=a$ for some $c\in\C_+$, then $\b_a(\theta)/b_c$ is a base element, where $\b_a$ is the M\"obius transform of the unit disk with
zero at $a$,
$$\b_a=\frac{z-a}{1-\bar a z}.$$
A general description of the set $\DD_B(\theta)$ in terms of $\theta$ is an important but difficult problem. As we  will see in Section \ref{secTOin}, it generalizes the problem of describing complete and minimal sequences of reproducing kernels in model and de Branges spaces.

\ms\no Let $I$ and $J$ be two inner functions. We say that $f\in N[\bar IJ]$ is purely outer if $f$ is outer
and
$$\bar I J f=\bar g$$
where $g$ is outer. Note that then automatically $f=g$.

\ms\no We call an element $I$ of $\DD(\theta)$ total, if $N[\bar \theta I]$ contains a purely outer function.
We chose this name for such elements because, in a sense, each total element represents a total inner component
of one of the functions from $N[\bar\theta]=K_\theta$. Indeed, if $If\in N[\bar\theta]$ for some inner $I$ and
outer $f$, then
$$\bar \theta If=\bar J\bar f$$
for some inner $J$. Then
$$\bar\theta IJf=\bar f$$
and therefore $N[\bar \theta IJ]$ contains a purely outer function and $IJ$ is a total element of $\DD(\theta)$.
Moreover, every total element can be obtained this way, i.e., it is a total inner component of a function from $N[\bar\theta]$,
combining inner components in both half-planes. We denote by $\DD_T(\theta)$ the subset of all total elements of $\DD(\theta)$.

\ms\no One can show that together with each  function $I$ the set $\DD_T(\theta)$ it contains every $J$ such that
$I/J$ is a finite Blaschke product.
It follows that $\DD_T(\theta)$ contains the set  of all  inner divisors $I$ of $\theta$ such that $\theta/I$ is a
finite Blaschke product. Finite products can be replaced is this statement
with all Blaschke products whose arguments $\psi$ satisfy $\psi/2\in \textrm{Log} |H^2|$.
Here we denote by $\textrm{Log} |H^2|$ the set of functions
$$\{f|\  f=\ln|g|, \ g\in H^2(\C_+)\}.$$ In other words,
$\textrm{Log} |H^2|$ consists of real functions $f$ such that $f\in L^1_\Pi$ and  $\exp(f_+)\in L^2$,
where $f_+=\max(f,0)$.

\begin{proposition}\label{p200}
	Every element of $\DD(\theta)$ is a factor of a total element.
\end{proposition}

\ms\no To prove the last statement just notice that if $\bar\theta I Jf=\bar L\bar f$ for some inner $J,L$ and outer $f\in H^2$, then  $JLI$ is the desired
total element.


\ms\no In regard to relations between our new sets we have

\begin{proposition}\label{pbt} For every inner $\theta$
$$\DD_B(\theta)\subset\DD_T(\theta)\subset\DD(\theta).$$
The sets $\DD_T$ and $\DD_B$ are equal iff $\theta$ is a Blaschke factor (in which case they are
also equal to $\DD(\theta)$ and consist of constants).
\end{proposition}

\begin{proof}
If $I$ is a base element then the relation $\bar \theta If=\bar J\bar f$ implies that $f$ is outer and $J$ is constant: otherwise $I\div JI\in\DD(\theta)$,
which contradicts that $I$ is a base element.
Hence, $f\in N[\bar\theta I]$ is purely outer and $I$ is also a total element. The second statement follows from the fact
 $\DD_T$ contains  base elements divided by any finite Blaschke sub-products. Notice that $\DD_B$ cannot consist
 of singular functions only.

\end{proof}

\ms\no Since together with every element $\DD(\theta)$ contains all of its inner divisors, Proposition \ref{p200} implies that $\DD(\theta)$ is
determined by the set of its total elements. The inverse statement follows from Theorem \ref{t500}:

\begin{corollary}\label{cor200}
	If $I\tsim J$ then $\DD_T(I)=\DD_T(J)$.
\end{corollary}

\ms\no To deduce the above corollary note that total elements of $\DD$ are the total inner components of the de Branges space. The spaces $B(E_I)$ and
$B(E_J)$ must coincide as sets by Theorem \ref{t500}.

\subsection{$\DD(\theta)$ and $\DD_T(\theta)$ in terms of arguments}

It is not difficult to describe elements of $\DD_T(\theta)$ in terms of arguments. Let us start with $\DD_T(\theta)$
in the case when $\theta $ is a MIF. In this case all functions from $\DD_T(\theta)$ are MIFs and their arguments
are real-analytic functions on $\R$ (defined uniquely up to $2\pi n$). Recall the notation $\phi(I,J)=\frac 12(\arg I-\arg J)$.

\begin{proposition}\label{p100}
	$$I\in\DD_T(\theta)\Leftrightarrow \ti\phi(\theta,I)\in \textrm{Log} |H^2|.$$
	
\end{proposition}

\ms\no As to $\DD(\theta)$, recall that it consists of all divisors of functions from $\DD_T(\theta)$.

\begin{corollary}
	$I$ belongs to $\DD(\theta)$ iff
	$\phi(\theta,I)=\ti h+\frac 12\alpha$ where $h\in\textrm{Log} |H^2|$
	and $\alpha$ is an argument of an inner function.
	
\end{corollary}

\ms\no To establish the above statement, simply notice that $\phi-\frac 12\alpha$ for some argument of a MIF $\alpha$ is the argument of a purely outer element of $N[\bar \theta I]$.

\ms\no With some additional effort one can find analogs of statements from this section for general (non-MIF) inner functions.

%
%
%

\subsection{Total elements and de Branges' Theorem 66}

\ms\no A well known theorem by L. de Branges, Theorem 66
from \cite{dBr}, page 271, is an important result in the area of UP. One can
find a discussion of this result and its applications in \cite{SY}, along with further references.

\ms\no More general versions of this theorem from \cite{GAP, Type} played important roles in
the study of the Gap and Type problems, see also \cite{CBMS}. Here we present the statement from \cite{GAP}[Corollary 4]
in the settings of TO.

 \begin{theorem}\label{thm66}

Let $I,\theta$ be inner functions in $\C_+$, $\theta\in\DD(I)$.

\ms\no Then there exists an inner function $J$ in $\C_+$ such that $\spec_J\subset\spec_I$ and
$\theta\in \DD_T(J)$.

\ms\no The function $J$ can be chosen so that the purely outer $f\in N[\bar J\theta]$
is also zero-free on $\R$.
If $\theta$ is a meromorphic function, then $J$ can be chosen as a meromorphic function.

\end{theorem}

\ms\no If $I$ is a MIF, then $f$ in the statement is analytic through $\R$ and the term 'zero-free' can be understood
in the usual sense. In the general case, a function $f\in H^2$ has a zero at $x\in\R$ if $f/(z-x)\in H^2$, and a
zero-free function has no such points.

\ms\no Let us finish this section with the following problem. Given a collection of inner functions, we will call the minimal $\DD(\theta)$ containing these functions the Toeplitz hull (TH) of our collection. It seems to be an interesting question to find TH for a given collection. In view of our discussion in Section \ref{secTOin}, versions of this problem are equivalent to finding the minimal de Branges space or model space for a given collection of zero sets, etc.

\section{TO and TE in comparison with other relations}

\subsection{TO as a proper extension of the order by division}\label{secdiv}
As was mentioned before, another natural way to introduce a partial order on the set of inner functions is by division. We say that an inner function
$I$ divides another inner function $J$ if $J/I$ is an inner function. The relation 'divides' satisfies the axioms of a partial order.
Toeplitz order introduced above is an \textit{extension} of the order by division, i.e. if $I$ divides $J$ then $I\tleq J$.

\ms\no TO is a proper extension because one can easily construct a pair of inner functions $I$ and $J$ such that $I\tleq J$
but $J$ does not divide $I$. Indeed, choose any pair such that $J$ divides $I$ and $J$ has at least one zero.
Then that zero has also to be a zero of $I$. Take that zero of $I$ and move it by a finite distance in $\C_+$. It is not difficult to show (an exercise on Toeplitz kernels) that
then we still have $I\tleq J$, although $J$ no longer divides $I$.

%

\subsection{TO versus dominance}
 Intuitively, when $N[\bar I J]\neq 0$ for two inner
functions $I$ and $J$ it means that $I $ is 'larger' than $J$. This relation between $I$ and $J$ starts to resemble a strict order
even more after one recalls that by a lemma of Coburn $N[\bar I J]$ and $N[\bar J I]$ cannot be non-trivial simultaneously. Formally, however, this
relation does not constitute an order due to the lack of transitivity: $N[\bar I J]=N[\bar  JL]=0$ does not imply $N[\bar I L]=0$.

\ms\no Accordingly,
the relation $I\asymp J$, which can be defined to mean that  $N[\bar I J]=0$ and $N[\bar J I]=0$, fails to produce a formal equivalence.
An interesting geometric connection for this relation is observed in \cite{ACL}. It is shown that for two inner
functions $I\asymp J$ holds if and only if the subspaces $IH^2$ and $JH^2$, viewed as points in the Grassmanian manifold of all closed
subspaces of $L^2$, are connected by a geodesic. Lack of transitivity for this relation can be illustrated by the following example.

\begin{example}\label{extwo} Let us construct three MIFs $I, J$ and $L$ such that $I\asymp J$ and $J\asymp L$ but $I\not\asymp L$, where the relation $'\asymp'$
is defined as above.

\ms\no Let $C>0$ be a large number and let $I$ be a Blaschke product with zeros at $n +iC,\ n\in \Z$. Let $J$ be the Blaschke product with zeros at
$n +iC$ for $n< 0$ and at $(n+\frac 12) +iC$ for $n\geq 0$. Finally, let $L$ be the Blaschke product with zeros at $n +iC,\ n\in \Z, n\neq 0$.

\ms\no Then $\psi=2\phi(J,I)=\arg J -\arg I$ tends to 0 as $x\to-\infty$. For large positive $x$, $|\psi(x)+\frac\pi 2|<\e$, where $\e=\e(C)$ is a small number,
$\e(C)\to 0$ as $C\to\infty$.
From basic properties of Toeplitz kernels, since
$$\limsup_{x\to -\infty}\psi(x)-\liminf_{x\to \infty}\psi(x)<\pi\textrm{ and }$$
$$\limsup_{x\to -\infty}-\psi(x)-\liminf_{x\to \infty}-\psi(x)<\pi,$$
both $N[\bar I J]=0$ and $N[\bar J I]=0$, i.e., $I\asymp J$. Similarly, $\psi=\arg L-\arg J$
tends to 0 as $x\to-\infty$ and $|\psi(x)+\frac\pi 2|<\e$ near $\infty$, which implies $J\asymp L$.

It is left to notice that $N[\bar I L]=N[\bar b_{Ci}]$, where $b_{Ci}=-\frac{z+Ci}{z-Ci}$ is the Blaschke factor, and the kernel contains an $H^2$-function
$\frac 1{z-Ci}$. Hence $I\not\asymp L$.

\end{example}

\ms\no To study the relations between TO and triviality of kernels further let us formulate the following
statement, showing in particular that TO is an extension of the order mentioned at the end of Section \ref{secEx}.

\begin{lemma}\label{lemdom}
Let $I_1$ and $I_2$ be non-constant inner functions such that $N^\infty[\bar I_1 I_2]\neq 0$. Then $I_1\tgeq I_2$.

\end{lemma}

\begin{proof} If $\theta\in \DD(I_2)$ then for $f\in N[\bar I_2\theta]$ and $g\in N^\infty[\bar I_1 I_2]$ we have
$$\bar I_1 \theta g f=(\bar I_1 I_2  g)(\bar I_2\theta f)\in\bar H^2.$$
Therefore $\theta\in \DD(I_1)$.

\end{proof}

\subsection{TE versus twins}\label{secTEvst}
Following \cite{MIF1}, we will call two MIFs $I$ and $J$ twins if $\spec_I=\spec_J$.
This relation naturally appears in applications to spectral problems involving isospectral differential operators.

\ms\no Clearly, twin relation is an equivalence relation
on the set of all MIFs, which is different from the Toeplitz equivalence. It is obvious that $I\tsim J$ does not imply that $I$ is a twin
of $J$. The opposite implication fails in general as well. However, we do have the following statement. We use the notation
$f\asymp g$ for two functions $f$ and $g$ if $c|f|<|g|<C|f|$ for some positive constants $c$ and $C$ and all values of the argument.

\begin{lemma}\label{lemcomp} Let $I$ and $J$ be two MIF twins with the common spectrum $\sigma\subset \hat \R$. Then $I\tsim J$ iff
$I'\asymp J'$ on $\sigma$.

\end{lemma}

\begin{proof}
Let $\mu=\alpha_\infty\delta_\infty + \sum \alpha_n \delta_{x_n}$
and $\nu=\beta_\infty\delta_\infty + \sum \beta_n \delta_{x_n}$ be the Clark measures of $I$ and $J$ respectively.

\ms\no Suppose first that $I'\not\asymp J'$ on $\sigma$. Then by the formula for pointmasses of Clark measures given
in Section \ref{secMod}, $\alpha_n\not\asymp \beta_n$. WLOG, we can assume that there exists $f\in K_I \ (f\in L^2(\mu))$
such that $f\mu\neq h\nu$ for any $h\in L^2(\nu)$. Let $\theta\in \DD(I)$ be the total inner component of $f\in K_I$.  We can assume
that it is the inner component of $f$ in $\C_+$. Since $ f=(1-I)Kf\mu$,  $\theta$ is the inner factor of the Cauchy integral $Kf\mu$. If $I\tsim J$ then $\theta\in \DD(J)$ and there exists $g\in K_J$ such that $Kg\nu$ is divisible
by $\theta$ in $\C_+$. Moreover, by Corollary \ref{cor200}, $g$ can be chosen so that $\theta$ is its total inner component. Then $Kg\nu/Kf\mu$ is an entire function of exponential type zero without zeros. Hence it is a constant, which implies $f\mu=\const \cdot g\nu$ and we have a contradiction. 

\ms\no It is left to notice that if $I'\asymp J'$ then $L^2(\mu)=L^2(\nu)$, which implies $\DD(I)=\DD(J)$ and $I\tsim J$.

\end{proof}

\subsection{TE and invertibility}
Another important relation between inner functions, which resembles equivalence, comes from invertibility
of the Toeplitz operator with the symbol $\bar IJ$.  Due to the work of Hruschev, Nikolski,  and Pavlov \cite{HNP}, this condition became one of the main tools in the study of basis properties for systems
of reproducing kernels, including the classical problem on exponential bases as a particular case. Up to some technical details,
a system of reproducing kernels $\{k_\l\}_{\l \in\L}$ forms a Riesz basis in a model space $K_I$ if and only if $T_{\bar I B_\L}$ is invertible.

\ms\no Intuitively, the condition that $T_{\bar IJ}$ is invertible also tells us that the functions
$I$ and $J$ are similar. This relation is reflexive as $T_{\bar IJ}$ is invertible iff $T_{\bar JI}$ is.
Our next goal is to show that Toeplitz equivalence is not the same as the invertibility of $T_{\bar IJ}$. As a matter of fact,
unlike Toeplitz equivalence, invertibility is not a formal equivalence since, once again, it lacks transitivity.

\begin{example}\label{ex1} Similar to Example \ref{extwo}, construct $I_1,I_2,I_3$ so that the difference of arguments $$\arg I_{n+1} -\arg I_{n},\ n=1,2,$$ is smooth and close  to $\pi/3$ at $\infty$ and to $-\pi/3$ at $-\infty$.
Then $T_{\bar I_nI_{n+1}},\ n=1,2,$ is invertible but $T_{\bar I_1 I_3}$ is not, as follows from a theorem by Devinatz and Widom. Thus invertibility does not induce an equivalence relation.
\end{example}

\section{Conditions for TE}\label{secCon}

\ms\no While we do not see a reasonable 'if and only if' condition which describes TE in terms of the arguments or other requisites of inner functions,
here we give some simple 'one-sided' conditions for two MIFs to be equivalent. Recall that for two inner functions $I$ and $J$
we denote by $\phi=\phi(I,J)$ the function $\phi=\frac 12(\arg I-\arg J)$. If $I\tsim J$ then $\ti \phi$ is Poisson summable
and $h=h(I,J)$ stands for the outer function $h=e^{\ti\phi -i\phi }$, $|h|=\exp{\ti\phi}$.

\ms\no As was mentioned before, if $\ti\phi(I,J)$ is bounded, i.e., $|h(I,J)|$ is bounded and
separated from zero on $\R$, then $I\tsim J$. This condition is not necessary as was shown in Example \ref{exone}.

\subsection{Conditions in terms of arguments and derivatives}
Our first necessary condition is in terms of the difference of arguments and derivatives.

\begin{lemma}\label{nec}
Let $I$ and $J$ be two MIFs, $I\tsim J$.  Then
$$\frac {|J'|}{|I'|}\exp{2\ti\phi}\asymp 1$$
on $\R$.

\end{lemma}

\begin{proof}
By Theorem \ref{thmmult} multiplication by $h=h(J,I)$ is a bounded operator $K_J\to K_I$. Hence,
$$ |h(x)k^J_x(x)|=|<k^I_x,hk^J_x>_{K_I}|\leq ||k^I_x||_{K_I} \ ||hk^J_x||_{K_I}\leq
$$$$\leq C|I'(x)|^{1/2}||k^J_x||_{K_J}=
$$$$= C|I'(x)|^{1/2}|J'(x)|^{1/2}\leq C |J'(x)|\left(\frac {|I'|}{|J'|}\right)^{1/2}=C|k^J_x(x)|\left(\frac {|I'|}{|J'|}\right)^{1/2},$$
for all $x\in \R$, which implies one of the two estimates. Applying similar argument to the operator $K_I\to K_J$ we obtain the other.

\end{proof}


Further metric properties of $h$ give the following conditions.

\begin{theorem}
Let $I,J$ be MIFs, $\phi=\phi(I,J)$.

\ms\no If  the functions $|J'|^{1/2}\exp{(-\ti\phi)}$ and $|I'|^{1/2}\exp{(\ti\phi)}$ belong to $L^2(\R)$ then $I \tsim J$.

\ms\no If  $I\tsim J$ then $\ti\phi-\log(1+|x|)\in \textrm{Log} |H^2|$.

\end{theorem}

\begin{proof}
By Theorem \ref{tto}, $I \tsim J$ iff multiplication by $h(I,J)$ is a bounded invertible operator from $K_I$ to $K_J$.
Note that since every $f\in K_I$ satisfies $|f|\leq ||f||_2|I'|^{1/2}$ the conditions in the statement
imply that $hf\in H^2$ and therefore $hf\in K_J$. Similarly, for every $f\in K_J$, $f/h\in K_I$.

\ms\no Note that two singular MIFs cannot be equivalent unless they are constant multiples of each other. Hence,
if $I\tsim J$ then one of them, say $I$, has a zero. If $I(a)=0$ then $I/b_a$ is a base element of
$\DD(I)$, and therefore a base
 element of $\DD(J)$. By Proposition \ref{pbt},
it is a total element of $\DD(J)$ and by Proposition \ref{p100}, $\ti\phi(J,I/b_a)\in \textrm{Log} |H^2|$. It is left to notice
that $\ti\phi(J,I/b_a)\sim \ti\phi(I,J)-\log(1+|x|)$ as $x\to \pm\infty$.

\end{proof}

\ms\no The following question was suggested by the referee in relation to the above proof: Is it true for general singular inner functions that they are equivalent iff they are constant multiples of each other? We would like to leave this question for the reader.

\subsection{TE for functions with comparable derivatives}
The condition of comparability for the derivatives of the inner functions appearing in Section \ref{secTEvst} is worth exploring a bit further.
Such conditions appear in applications. For instance, inner functions corresponding to Schr\"odinger equations with regular potentials, as
well as to other similar classes of canonical systems, will satisfy this condition. Completeness problems for various families of special functions
also lead to MIFs with comparable derivatives, see \cite{MIF1}. Let us provide the following description of
Toeplitz equivalence pertaining to this case.

\begin{lemma} Consider two MIFs $I$ and $J$ such that $I'\asymp J' $ on $\R$. Then $I\tsim J$ iff $\phi(I,J)=\frac 12(\arg I -\arg J)$ has a bounded harmonic conjugate.
\end{lemma}

\begin{proof}
If $\ti\phi$ is bounded then $I\tsim J$, see Section \ref{secCon}. Assume now that $I\tsim J$ but $\ti\phi$ is unbounded.
This contradicts Lemma \ref{nec}.

\end{proof}

\section{TO in terms of model  and de Branges spaces}\label{secTOin}

\subsection{ $\DD(\theta)$ as inner factors in $K_\theta$ and $B(E)$}

In terms of the model space $K_\theta$, the set of dominance $\DD(\theta)$ has a natural meaning.
It is the set of all  inner components of functions from $K_\theta$.

\ms\no In case of MIFs, $K_\theta$ is directly related to the de Branges space $B(E)$ via
the isometric isomorphism $EK_\theta =B(E)$. Hence, $\DD(\theta)$ is also
the set of all  inner components of functions $f/E$, $f\in B(E)$, in the upper  half-plane.

\ms\no The set $\DD_T(\theta)$ can be similarly identified with the subset of all total inner components
of functions from $K_\theta$ or $B(E)$ as was discussed in Section \ref{secB}.

\ms\no If $\theta$ is a MIF and  $I\in \DD_B(\theta)$ then $I=B_\L S^a$ for some Blaschke sequence
$\L,\ \lan\to\infty$ and $a\geq 0$. In the case of pure Blaschke product, $a=0$, the sequence $\L$ satisfies $\L=\{f=0\}$ (with muliplicities)
for some function from $K_\theta$ (or $B(E)$). In the case $a>0$, for any M\"obius transform $\b_w$ of the unit disk, $\b_w(S^a)B_\L$ is a Blaschke product from $\DD_B(\theta)$.
Hence, $\L\cup \{\frac{2\pi n}a+iC\},\ \Re C>0,$ is again equal to $\{f=0\}$ for some $f\in K_\theta\ (B(E))$.

\ms\no We will return to the discussion of zero sets in Section \ref{secCom}.

\ms\no Sets of inner components of functions from $K_\theta$ have been studied
by other authors, see for instance \cite{Dya, BS}. As follows from our discussion above,  $I\in\DD(\theta)$ iff '$I$ lurks within $K_\theta$', using the terminology of \cite{BS}. 
In \cite{FHR} the authors study the set of multipliers $\MM(I,J)$ between model spaces $K_I$ and $K_J$, i.e., the set of
all $H^\infty$-functions $\phi$ such that $\phi K_I\subset K_J$. In relation to TO, $\MM(I,J)\neq \{0\}$ implies
$I\tleq J$. In these papers the reader may find additional properties of $\DD(\theta)$.

\subsection{Base elements and asymptotics along the imaginary axis}

\begin{lemma}\label{lemtot}
Let $\theta$ be a MIF, $f\in K_\theta$, $f(iy)\neq o(y^{-3/2})$ as $y\to\infty$.
Then the total inner component of $f$ is a base element of $\DD(\theta)$.

\end{lemma}

\begin{proof}
Suppose that the total inner component $I$ of $f$ is not a base element.
Then there exists inner $J$ and outer $g$ such that $I$ properly divides $J$ and $Jg\in K_\theta$.
Let $h$ be
an outer component of $f$. Then the argument of the outer function $g/h$ is $-\frac 12\arg (J/I)$, i.e., it is
a continuous decreasing function on $\R$ which decreases by at least $\pi$. By Claim \ref{c1} below and the asymptotics of
$f$ this implies that $g(iy)\neq o(y^{-1/2})$ as $y\to\infty$. This contradicts $g\in H^2$.

\end{proof}

\ms\no The following can be easily established.

\begin{claim}\label{c1}
Let $h$ be an outer function in $\C_+$ whose argument $\psi$ on $\R$ satisfies
$$\liminf_{x\to -\infty} \psi(x) -\limsup_{x\to\infty}\psi(x)\geq \pi.$$
Then
$$y=O(h(iy))\textrm{ as }y\to\infty.$$

\end{claim}

\ms\no Our next statement combined with Lemma \ref{lemtot} shows that functions whose total inner components
are base elements of $\DD(\theta)$ are dense in $K_\theta$.

\begin{proposition}\label{propass}
For every inner $\theta$, the space $K_\theta$ contains a dense subset of functions $f$ satisfying
$$|f(iy)|\sim \frac 1y\textrm{ as }y\to\infty.$$

\end{proposition}

\begin{proof}
Let $C(z)$ be the Cayley transform from the unit disc to the upper half-plane. Then $\Phi(z)=\theta(C(z))$ is
an inner function in the unit disc. Recall that $K_\theta$ is obtained from $K_\Phi$ via the map
$f(z)\mapsto (C^{-1}(w)-1)f(C^{-1}(w))$. Now the statement is equivalent to the statement that functions with
finite non-zero limits $\lim_{r\to 1-}f(r)$ are dense in $K_\Phi$.

\ms\no Let $\Phi_n$ be a sequence of divisors of $\Phi$ such that  $\Phi_n\to\Phi$ point-wise in $\D$ and
each $\Phi_n$ can be analytically continued through $1$ . Then $\cup K_{\Phi_n}$ is dense
in $K_\Phi$. But in each $K_{\Phi_n}$ all functions can be continued through $1$ and a dense subset have non-zero values there.

\end{proof}


\subsection{$\DD(\theta)$ in terms of de Branges spaces.}\label{sec3.1}

Let $E$ be a de Branges function and let $\theta=\theta_E$ be a corresponding MIF.
If $F\in B(E)$ then $F=I_1fE$ in $\C_+$, where $I_1$ is inner and  $f\in H^2$ is outer. Similarly,
in $\C_-$, $F=\bar I_2\bar f E^\#$. An important property of $B(E)$ is that the inner components
can be moved from one half-plane to the other, i.e., if $I_3I_4=I_1$ then the function $G$ defined
as $I_3fE$ in $\C_+$ and as $\bar I_4\bar I_2\bar fE^\#$ in $\C_-$ also belongs to $B(E)$. Similarly
one can move inner factors from $\C_-$ to $\C_+$.

\ms\no The set of all inner functions $I_1$ ($I_2$) appearing this way for a fixed $B(E)$ is exactly the dominance set $\DD(\theta)$.

\ms\no If $F$ is an entire function defined as above in $\C_\pm$, we will call the inner function $I_1I_2$ the total inner component of $F$.
If $I$ is a total inner component for a function from a de Branges space then the argument of $fE$ on $\R$ is determined by the argument
of $I$ up to $\pi n$. The argument of MIF $I$ is a real-analytic function on $\R$, while the argument of $fE$ is piece-wise real analytic, making
a jump of $-\pi$ at each real zero of $fE$. All in all we have
\begin{equation}\arg fE= \frac 12\arg I\ (\textrm{mod}\ \pi)\label{e3.1}\end{equation}

\ms\no Note that total inner components of functions from $B(E)$ are exactly the elements of $\DD_T(\theta_E)$.

\ms\no Denote by $\DD_T^*(\theta_E)$ the set of {\it exact total elements}, the total elements corresponding to functions from $B(E)$ which have no zeros on the real line.
If $f\in B(E)$ is such a function and $I\in \DD_T^*$ is its inner component in $\C_+$ then the last equation holds exactly, i.e.,
$\arg fE= \frac 12\arg I$ on $\R$ for a proper choice of arguments on both sides.

\subsection{dB spaces for inner Toeplitz kernels}\label{secdBs}

The following discussion will be used in Section \ref{secCom}.

\ms\no Let $I$ and $J$ be two MIFs such that $N[\bar I J]\neq 0$. Notice that
$$\bar I J f=\frac{E_I}{E_I^\#}\frac{E_J^\#}{E_J}f=\bar g$$
which shows that an $H^2$-function $f$ belongs to $ N[\bar I J]$ iff $\frac{E_I}{ E_J} f$ can be continued to the lower half-plane as an entire function
(the formula for the continuation is $\frac{E_I^\#}{E_J^\#}\bar g$). Consider the space
of entire functions $B=\frac{E_I}{ E_J}N[\bar I J]$ equipped with the norm
$$||f E_I/E_J||=||f||_{H^2}.$$
By verifying the axioms one can conclude that $B=B(E)$ is a de Branges space for some HB function $E$.
We will denote this HB function by $E_{I,J}$.

\ms\no To summarize, to each pair of MIFs $I,J$ such that
$N[\bar I J]\neq 0$ there corresponds an HB function $E_{I,J}$. Our construction
implies the following important property:

\begin{proposition}\label{prop103}
The set $J\DD(\theta_{E_{I,J}})$ is the set of all functions from $\DD(I)$ divisible by $J$

\end{proposition}

\subsection{TE and equality of de Branges spaces}

While model spaces $K_\theta$ are equal as sets if and only if the corresponding inner functions
are equal up to a constant multiple, de Branges spaces $B(E)$ and $B(\ti E)$ can be equal as sets for two different functions $E$ and $\ti E$.

\ms\no Equality of two de Branges spaces as sets of functions, with (possibly) different norms, is an important aspect of spectral
theory for differential equations. The so-called Gelfand-Levitan theory which treats spectral problems for regular Schr\"odinger
equations and Dirac systems utilizes the fact that the corresponding de Branges spaces are equal to Paley-Wiener spaces as sets.
This property becomes the key ingredient of the theory allowing one to use the structure of Paley-Wiener spaces to study relations
between the potential of the  differential operator and the Fourier transform of its spectral measure. An extension of
Gelfand-Levitan techniques to more general classes of Krein's canonical systems, see for instance \cite{Etudes}, requires further
understanding of properties of HB functions $E$ and $\ti E$ which produce equal, as sets, spaces $B(E)$ and $B(\ti E)$. Such
questions are also equivalent to problems on sampling measures, see Section \ref{secS}.

\ms\no Although total description of such pairs of HB functions presents
an important open problem, intuitively such functions must be similar to each other, which raises a natural question
on the correspondence of this  relation and Toeplitz equivalence for the MIFs $\theta$ and $\ti \theta$. Our next theorem
connects this problem to TO.

\ms\no We will use the notation $B(E)\seq B(\ti E)$ for the two de Branges spaces equal as sets.
Note that if $B(E)\seq B(\ti E)$ then  norms in the spaces are automatically equivalent.

\begin{theorem}\label{t500} Let $E$ and $\ti E$ be HB functions such that $E/\ti E\in \NN(\C_+)$. Then
$\theta\tsim\ti\theta$ for the corresponding MIFs  iff $B(E)\seq B(\ti E)$.

\ms\no Conversely,
if $\theta\tsim\ti\theta$ for two MIFs $\theta$ and $\ti \theta$ then the corresponding
HB functions can be chosen to satisfy $E/\ti E\in \NN(\C_+)$ and $B(E)\seq B(\ti E)$.
\end{theorem}

\begin{proof} Suppose first that $B(E)\seq B(\ti E)$. Since $\DD(\theta)$ and $\DD(\ti\theta)$ are the sets of  inner
components of $F/E$ for the elements $F$ of the corresponding space in $\C_+$, $\DD(\theta)=\DD(\ti\theta)$
and $\theta\tsim\ti\theta$.

\ms\no Conversely, let $\DD(\theta)=\DD(\ti\theta)$. Then the subsets of base elements, $\DD_B$, coincide as well.
If $I\in\DD_B(\theta)=\DD_B(\ti\theta)$ then $I$ is a total inner component for some $F\in B(E)$ and for some $G\in B(\ti E)$.
Note that then $F/G$ is a zero-free entire function. Indeed, since the total zero components of $F$ and $G$ coincide, $F/G$ may
only have zeros on the real line. Then $F$ has zeros on the real line, say at $a\in\R$. But then $(z-i)\frac F{z-a}$ is an element
of $B(E)$ with total inner component $b_iI$ which contradicts the property that $I$ is a base element. Hence $F/G$ is zero-free.
It must be outer in both half-planes because otherwise $I$ is not a base element in
one of the $\DD$ sets. Hence, $F/G=const$. We obtain that the sets of functions in $B(E)$ and $B(\ti E)$ whose total inner components
are base elements coincide.

\ms\no Let now $F\in B(E)\setminus B(\ti E)$. By Proposition \ref{propass}, there exists $H\in B(E)$ such that $(H/E)(iy)\neq o(y^{-3/2})$ as $y\to\infty$.
By Lemma \ref{lemtot}, the total inner component of $H$ is a base element and therefore $H\in B(\ti E)$ by the argument above.

\ms\no Notice that $(F/E)(iy)=o(y^{-3/2})$, because otherwise its total inner component would have been a base element which would imply $F\in B(\ti E)$.
Hence, $(F+ H)/E\neq o(y^{-3/2})$, which implies that $F+ H$ has a base total inner component and therefore belongs to both de Branges spaces.
Since $H$ also belongs to both spaces, so does $F$ and we arrive at a contradiction.
\end{proof}

\ms\no In the process of the last proof we have established the following useful property.

\begin{proposition}
If the total inner component $I$ of a function $F$ from $B(E)$ is a base element of $\DD(\theta_E)$ then $F$ has no
real zeros. Equivalently, $\DD_B(\theta)\subset \DD_T^*(\theta)$ for any MIF $\theta$.
\end{proposition}

\subsection{TE and model spaces}

In this section we formulate our result for general inner functions. In order to do that we will need to extend the notations
$\phi(I,J)$ and $h(I,J)$ introduced in Section \ref{secdif} from the case of MIFs to the general case.

\ms\no To make sense of the definition $\phi(I,J)=\frac 12(\arg I - \arg J)$ in the general case we  understand $\arg I \ (J)$ as
a measurable function on $\R$ such that $I/e^{i\arg I}$ is positive a.e. on $\R$. It is not difficult to show that if $I\tsim J$
then their arguments can be chosen in such a way that $\ti\phi$ exists and  $h(I,J)=e^{\ti\phi-i\phi}$ is an outer function.
In what follows we will assume that $\phi$ and $h$ correspond to the arguments of $I$ and $J$ chosen in such a way.

\ms\no Our statement in this section is

\begin{theorem}\label{tto}\label{thmmult}
$I\tsim J$ iff multiplication by $h(I,J)$ is a bounded and invertible operator $K_I\to K_J$.

\end{theorem}

\ms\no For general $I$ and $J$ this means that  if $I\tsim J$ then their arguments can be chosen so that the outer function $h(I,J)$ exists and multiplication by $h(I,J)$ is a bounded and invertible operator $K_I\to K_J$. Conversely, if the arguments can be chosen in such way, then $I\tsim J$.

\begin{proof}
Suppose first that $I$ and $J$ are MIFs. Then $h=E_I/E_J$ and the equivalence of $I\tsim J$ and $B(E_I)\seq B(E_J)$ gives the statement.

\ms\no In the general case, if multiplication by $h(I,J)$ is a bounded and invertible operator $K_I\to K_J$
then the sets of all inner components of functions from $K_I$ and $K_J$ coincide because $h$ is outer. Hence $\DD(I)=\DD(J)$ and $I\tsim J$.
In the opposite direction, if $I\tsim J$ one can reduce the proof to the case of MIFs via a limiting argument. We leave the details to the reader.
\end{proof}

\subsection{Complete and minimal sequences of reproducing kernels and zero sets in de Branges spaces}\label{secCom}\label{secC}

\ms\no We call a sequence $\L\subset\C$  a zero set of a de Branges space $B(E)$ iff there exists $f\in B(E),\ f\not\equiv 0,$ such that $f=0$ on
$\L$ (with multiplicities). We call $\L$ an exact zero set if there exists $f\in B(E)$ such that $\{f=0\}=\L$ (with multiplicities). 
A maximal zero set  is a sequence $\L$ of points such that there exists
a non-zero function from the space vanishing on $\L$, but there is no such function for any set properly containing $\L$.

\ms\no A maximal zero set is exact but not vice versa. Blaschke products corresponding to maximal zero sets are base elements from $\DD(\theta)$
and those corresponding to exact zero sets are total elements, see Lemma \ref{lem200} below. Maximal zero sets are also related to complete
and minimal sequences.

\ms\no We say that $\L$ is a complete
and minimal sequence for $B(E)$ iff the system of reproducing kernels $\{k_\l\}_{\l\in\L}$ is complete and minimal
(i.e., any proper subsequence is incomplete) in $B(E)$. Note that a sequence is complete and minimal iff the same sequence minus any one of its points
is a maximal zero set.

\ms\no For sequences $\L\in\C_+$ similar definitions can be given for the model spaces $K_\theta$.

\ms\no We will now establish relations between zero sets and the subsets of the dominance set.

\ms\no Recall that, as was defined in Section \ref{secMod}, the spectrum $\spec_I$ of a MIF $I$ is the sequence of points from $\hat \R$ where the function is
equal to 1.
If $\L\subset \C\setminus \R$ is a sequence of complex points we denote by $B_\L$ the Blaschke product with zeros
at the points of $\L$ in $\C_+$ and at the points conjugate to the points of $\L$ in $\C_-$, assuming the Blaschke condition holds.

\begin{lemma}\label{lem200} Let $E$ be an HB function, $\theta=\theta_E$. Let $\L\subset \C\setminus \R$ and $\G\subset\R$ be sequences of points.

\ms\no 1) $\L\cup \G$ is a zero set of $B(E)$ iff there exists an inner $I$ such that $\spec_I=\G$ and $B_\L I\in \DD(\theta)$;

\ms\no 2) $\L\cup \G$ is an exact zero set of $B(E)$ iff there exists an inner $I$ such that $\spec_I=\G$ and $B_\L I\in \DD_T^*(\theta)$;

\ms\no 3) $\L\cup \G$ is a maximal zero set of $B(E)$ iff there exists an inner $I$ such that $\spec_I=\G$ and $B_\L I\in \DD_B(\theta)$.

\end{lemma}

\begin{proof}

1) Suppose that $F=0$ on $\L\cup \G$ for some $F\in B(E)$. Then there exists a finite positive measure $\mu$ concentrated on $\G$,
$$\mu=\sum a_n\delta_{\gamma_n},$$
such that $a_n>0$ are small enough to satisfy $(F/E)K\mu\in H^2$. Then for $I=I_\mu$ (the inner function whose Clark measure is $\mu$) we have
$(F/(E(1-I))\in H^2$ and $G=F/(1-I)\in B(E)$. 
For the function $G/E\in K_{\theta_E}$ we have
$$\bar \theta_E \frac GE=\bar\theta_E\frac FE\frac 1{1-I}= \bar h \frac {\bar I}{1-\bar I}$$
a.e. on $\R$ for some $h\in H^2,\ \bar h=\bar\theta_E F/G$. Here we use the fact that $F/E\in  K_{\theta_E}$ and the observation that $\bar I(1-I)/(1-\bar I)>0$ a.e. on $\R$.
Since $F$ vanishes on $\L$, the inner component of $G/E$ is divisible by $B_\L$. According to the last equation,
the inner component of $\theta_E\bar G/\bar E$ is divisible by $I$. Hence the total inner component of $G/E$ is divisible by $B_\L I$.

\ms\no Conversely, let $B_\L I\in \DD(\theta)$ for some inner $I$ such that $\spec_I=\G$. Then $B(E)$ contains a function
equal to $B_\L I f E$ in $\C_+$, where $f$ is outer from $H^2$. Then $B(E)$ also contains a function
equal to $B_\L  f E$ in $\C_+$. Subtracting we obtain a function in $B(E)$
equal to $B_\L(1- I) f E$ in $\C_+$, which vanishes on $\L\cup \G$.

\ms\no 2) and 3) can be proved similarly

\end{proof}

\begin{theorem}\label{t200} Every  element of $\DD(\theta)$ is a divisor of a base element.
\end{theorem}

\ms\no  Before we prove the last statement let us note that each de Branges space
possesses a large collection of maximal zero sets (complete and minimal sequences, minus one point). For instance, if one takes an orthogonal basis of reproducing kernels described in Section \ref{secE} and deletes one point
from the corresponding sequence, the remaining sequence is a maximal zero set. By 'perturbing' this real sequence one can obtain a maximal zero set
in $\C_\pm$. Note that maximal zero sets $\L$ in $\C_+$, as any zero sets of a de Branges space $B(E)$ in $\C_+$, satisfy the Blaschke condition. The corresponding Blaschke products $B_\L$ are exactly the base elements of $\DD(\theta_E)$ which have no singular divisor.

\begin{proof}[Proof of Theorem \ref{t200}] First let us assume that $\theta$ is a MIF and let $J\in\DD(\theta)$.
Let $\L\subset \C_+$ be a  maximal zero set of $B(E_{\theta,J})$ as defined in Section \ref{secdBs}. Then $JB_\L$ is a base element of $\DD(\theta)$.
Indeed, if $b_aJB_\L\in \DD(\theta)$ for some Blaschke factor $b_a,\ a\in\C_+$ then by Proposition \ref{prop103},
$\L\cup \{a\}$ is a zero set of $B(E_{I,J})$, which contradicts maximality of $\L$. Hence $J$ divides a
base element of $\DD(I)$.

\ms\no Finally, in the case of non-MIF $\theta$, notice that $\theta$ is a normal limit of MIFs and apply a limiting argument.

\end{proof}

\ms\no Considering the case when $J$ is a Blaschke product in the statement of Theorem \ref{t200} and using Lemma \ref{lem200},
we obtain the following result by Yu. Belov. In fact, our last proof is  a variation of the proof in \cite{Belov}.

\begin{corollary}[\cite{Belov}]
Any incomplete sequence of reproducing kernels in a de Branges space is contained in a complete and minimal sequence
of reproducing kernels.
\end{corollary}

\ms\no (Note that any incomplete sequence of reproducing kernels is automatically minimal, which is used implicitly in the above statement.)

\ms\no Denote by $\ZZ(B(E))$ the collection of all zero sets for the space $B(E)$ and let $\ZZ_e$ stand for the exact zero sets.
Then Theorem \ref{t200} becomes the following statement.

\begin{theorem} The collection of zero sets $\ZZ(B(E))$ determines the space $B(E)$ uniquely within the regularity class of $E$, i.e., if $\ZZ(B(E))=\ZZ(B(\ti E))$
and  $E/\ti E\in \NN(\C_+)$
then $B(E)\seq B(\ti E)$.

\end{theorem}

\begin{proof}
The collection of sets $\L\setminus\R,\ \L\in \ZZ(B(E))$ determines the set of Blaschke products from $\DD(\theta), \ \theta=\theta_E$.
For the non-Blaschke elements we have the simple observation that whenever
$BS^a \in \DD(\theta)$, the Blaschke product $B {\textbf b}_w(S^a)$ belongs to $\DD(\theta)$ as well for all $w\in\D$, which implies that the sets $\L\setminus\R,\ \L\in \ZZ_e(B(E))$ determine $\DD(\theta)$ uniquely. It follows that $\ZZ_e(B(E))$ determines $\DD(\theta)$ and the statement follows from Theorem \ref{t200}.

\end{proof}

\ms\no Note that since $\ZZ(B(E))$ determines $B(E)$, it also determines $\ZZ_e(B(E))$. Conversely, since zero sets are subsets of exact zero sets, $\ZZ_e(B(E))$
determines $\ZZ(B(E))$. Even easier one can establish the same connection between $\ZZ$ and $\ZZ_m$, the collection of maximal zero sets, as
each maximal zero set is a maximal element of $\ZZ$ with respect to inclusion. Hence
either of the sets $\ZZ_e$ or $\ZZ_m$ can be substituted into the last statement instead of $\ZZ$. However, the statement with $\ZZ(B(E))$ is the strongest of the three.

\subsection{TO and inclusion of de Branges spaces}

\ms\no The last statement raises a natural question: if TE is equivalent to equality of the corresponding de Branges spaces,
are the relations $\tleq$ and $ \tgeq$ equivalent to inclusions of the spaces? If the answer were positive we would obtain
an equivalent definition of TO.

\ms\no The relation does hold in one direction:

\begin{proposition}\label{chain}
If $B(E)\subset B(\ti E)$  then $\theta\tleq\ti\theta$.

\end{proposition}

\ms\no The statement follows from the fact that the corresponding dominance set consists of all inner components of
$F/E, \ F\in B(E)$ ($F/\ti E,\ F\in B(\ti E)$) in $\C_+$.

\ms\no However, as  shown by the example below, the opposite direction fails.

\begin{example}\label{echain}
Consider a sequence $\lan=(2^{|n|}\sign\ n+\frac 12)\pi+\e_{|n|} i,\ n\in\Z$, where $\e_n\downarrow 0$, and the corresponding Blaschke product $B_\L$.
Denote $I=B_\L S$ (where, once again, $S(z)=e^{iz}$) and consider the corresponding Cartwright HB function $E_I$.

\ms\no Let $s(z)=\sin z/z$ be the sinc function. Then, if $\e_n$ decays to 0 fast enough, $s/E_I\not\in L^2$. Hence $s\not \in B(E_I)$
and $B(E_S)\not\subset B(E_I)$ because $s\in B(E_S)=PW_1$.
At the same time, since $S$ is a divisor of $I$, $S\tl I$.
\end{example}

\subsection{Chains of  de Branges spaces and chains in TO}

\ms\no Recall that for a partial order a chain is a subset where every pair of elements is comparable. On the other hand,
every de Branges space, or every Poisson finite measure on the real line, gives rise to a chain of de Branges
spaces of entire functions, see Section \ref{secE}. Although the term 'chain' is given different meanings in these two situations,
we note the following simple connections between de Branges chains and chains in Toeplitz order.

\ms\no It follows from Proposition \ref{chain} that de Branges chains produce chains in Toeplitz order: if $\{B(E_t)\}$ is a de Branges chain then
$\theta_{E_t}$ is a chain in TO. Clearly, such chains do not present all possible chains in TO since, for instance, not all such chains consist of MIFs. Even if we restrict our attention to all TO chains in the subset of MIFs, de Branges chains do not produce all such chains as follows from
Example \ref{echain}. Finding a way to determine if a chain in TO is a de Branges chain seems like another interesting problem.

\section{TO in problems of Harmonic Analysis}

\ms\no In the rest of the paper we look at connections of Toeplitz order with some of the classical problems of Harmonic Analysis.
Our goal here is to provide only a brief overview of such connections without going into deeper technical details.

\ms\no We start with two completeness problems for families of complex exponentials, the Beurling-Malliavin (BM) problem and its  extensions studied in \cite{MIF1, MIF2}, and
the Type problem recently considered in \cite{Type}. We then discuss sampling problems in Paley-Wiener and de Branges spaces with some remarks on the two-weight Hilbert problem, see \cite{OS} and \cite{NTV, L, L&Co} for further references.

\subsection{BM problem}

Let $\L=\{\l_n\}$
be a sequence of distinct points in the complex plane and let
$$E_\L=\{e^{i\l_n x}\}$$
be a sequence of complex exponential functions on $\R$ with frequencies from $\L$.

\ms\no For any complex sequence
$\L$ its
radius of completeness is defined as
$$R(\L)=\sup\{ a\ |\ E_\L \textrm{ is complete in }L^2(0,a)\}.$$
The famous BM problem which was solved in \cite{BM1, BM2}, asks to find a formula for $R(\L)$ for an arbitrary $\L\subset \C$.

\ms\no It is well-known in the theory of completeness that the general problem can be easily reduced to the case of real sequences $\L$.
More precisely, if $\L$ is a general complex sequence then $E_\L$ is complete in $L^2([0,a])$ if and only if
$E_{\L'}$ is complete in the same space, where $\L'$ is the real sequence defined as $\lan'=1/\Re \frac 1\lan$ (WLOG $\L$ has no purely imaginary points), see for instance
\cite{Koosis}.
Also, as will be explained below, one can always assume that $\L$ is a discrete sequence, i.e. has no finite accumulation points.

\ms\no A  system of complex exponentials $E_\L$ is incomplete in $L^2([0,a])$ if and only if there exists a non-zero $f\in L^2([0,a])$
such that $f\perp e^{i\l_n  x}$ for all $\lan\in\L$. Taking the Fourier transform of $f$ we arrive at the equivalent reformulation that
 $E_\L$ is incomplete in
$L^2([0,2a)])$ if and only if $\L$ is a zero for $PW_a$.

\ms\no One immediate consequence of this connection is that if $\L$ has a finite accumulation point then $R(\L)=\infty$.
Aso since any zero set $\L\subset\C_+$ of a $PW$-space must satisfy the Blaschke condition, $R(\L)=\infty$ for any non-Blaschke
$\L\subset\C_+$.

\ms\no To give the formula for $R(\L)$ we will need the following definitions.

\ms\no If $\{I_n\}$ is a sequence of disjoint intervals on $\R$, we call it short if
$$\sum\frac{|I_n|^2}{1+\dist^2(0,I_n)}<\infty$$
and long otherwise.

\ms\no If $\L$ is a sequence of real points define its exterior BM density (effective BM density)
as

$$D^*(\L)=\sup\{ d\ |\ \exists\textrm{ long  }\{I_n\}\textrm{ such that }\#(\L\cap I_n)\geqslant d|I_n|\},\ \forall n\}$$

\ms\no For a complex sequence define $D^*(\L):=D^*(\L')$.

\begin{theorem}[Beurling and Malliavin, around 1961, \cite{BM1, BM2}]\label{tBM}
	Let $\L$ be a discrete sequence. Then
	$$R(\L)=D^*(\L).$$
\end{theorem}

\ms\no In regard to Toeplitz order, BM theorem is equivalent to the following statements.
Recall that any MIF $I$ has the form $I=B_\L S^a$ where $B$ is a Blaschke product
and $S^a=e^{iaz}$ is the exponential function.  Put $r(I)=D^*(\L)+a$.

\ms\no The most direct equivalent of Theorem \ref{tBM} is in terms of the dominance set.

\begin{theorem}\label{tBMTO1}
	$I\in \DD(S^b)$ if $r<b$ and $I\not\in\DD(S^b)$ if $r>b$.
\end{theorem}

\ms\no The statement
can be equivalently reformulated in terms of TO.

\begin{theorem}\label{tBMTO2}
	For any MIF $I$,
	$$I\tleq S^b\ \Rightarrow\ r(I)\leq b$$
and
$$r(I)<b\ 	\Rightarrow\ I\tl S^b.$$
	
\end{theorem}

\ms\no Note that equivalence of the last two statement no longer holds if $S$ is replaced with a general inner function. Finding a broader set of functions
for which the equivalence does hold is an open problem.

\begin{proof}[Proof of Theorem \ref{tBMTO2}] The general case can be trivially reduced to the case $I=B_\L$. Suppose  first that $r=D^*(\L)>b$. Let $a\in \L$ be a zero of $I$.
	Then $D^*(\L\setminus\{a\})>b$ and $I/b_a\not\in \DD(S^b)$ by BM theorem (Theorem \ref{tBMTO1}). Since $I/b_a\in \DD(I)$, the relation $I\tleq S^b$ does not hold.
	
\ms\no 	To establish the second statement, suppose that $r=D^*(\L)<b$. If $J\in\DD(I)$ then there exists $f\in N[\bar IJ],\ f\not\equiv 0$.
	Also, since $D^*(\L)<b$, by Theorem \ref{tBMTO1} there exists $g\in N[\bar S^bI],\ g\not\equiv 0$. Note that $Ig\in S^{b/2}PW_{b/2}$ which implies
	$g\in H^\infty$. Then
	$$\bar S^b J gf=(\bar S^b I g )(\bar I J f)\in \bar H^2,$$
	which means that $J\in \DD(S^b)$. Hence $\DD(I)\subset\DD(S^b)$ and $I\tl S^b$.
	
\end{proof}

\ms\no As we can see, the Beurling-Malliavin formula gives a metric condition for the relation of TO in the very specific case when one
of the functions to be compared is the exponential function. Similar descriptions for more general classes
of inner functions, especially those appearing in applications to completeness problems and spectral analysis
remain mostly open. Below we present one of such extensions found in \cite{MIF2}.

 \subsection{Further generalizations}

 Reformulations of the BM theorem given in the last section present a clear direction for generalizations of the BM theory.
 While a statement analogous to Theorems \ref{tBMTO1} and \ref{tBMTO2} with a general inner function in place of $S^a$ may be out of
 reach at the moment, one can attempt to replace the exponential function with an inner function from a larger class.
To determine the right classes of inner functions to study in these settings one may look at a variety of applications
of the Toeplitz Approach in Harmonic analysis and Spectral Theory.

\ms\no One of such extensions was recently studied in \cite{MIF1, MIF2}. As was shown in \cite{MIF1} the class of MIFs with
polynomially growing arguments appears naturally in a number of applications including completeness problems for
Airy and Bessel functions, spectral problems for regular Schr\"odinger operators and Dirac systems, etc. An analog of
Theorem \ref{tBMTO1} proved in \cite{MIF2} can be applied to some of such problems. Here we present
an equivalent reformulation similar to Theorem \ref{tBMTO2}.

\ms\no  Let $\gamma: \R\to\R$
be a continuous function such that $\gamma(\mp\infty)=\pm\infty$.
i.e.,
$$\lim_{x\to-\infty}\gamma(x)=+\infty,\qquad \lim_{x\to+\infty}\gamma(x)=-\infty.$$
Define $\gamma^*$ to be the smallest non-increasing majorant of $\gamma$:
$$\gamma^*(x)=\max_{t \in [x,+\infty)}\gamma(t) .$$
The family of intervals $BM(\gamma)=\{I_n\}$ is defined as the collection  of the connected
components of the open set
$$\left\{x \in \R|~\gamma(x)\ne\gamma^*(x)\right\}.$$

\ms\no Let $\kappa\geq 0$ be a constant. We say that $\gamma$ is $\kappa$-{\it almost decreasing} if
\begin{equation}\label{kap}\sum_{I_n\in BM(\gamma)}(\dist(I_n,0)+1)^{\kappa-2}|I_n|^2<\infty.\end{equation}

\ms\no As before, an argument of a MIF $I$ on $\R$ is a real analytic function $\psi$ such that $I=e^{i\psi}$.

\begin{theorem}\label{tMIF2}
	Let $U$ be a MIF with $|U'|\asymp x^\kappa,\\ \kappa\geq 0$, $\gamma=\arg U $ on $\R$.
	Let $J$ be another MIF, $\sigma=\arg J$ on $\R$. Then
	
	\ms
	
\ms\no 	I) If $\sigma-(1-\e)\gamma$ is $\kappa$-almost decreasing, then $J\in \DD(U)$;
	
\ms\no 	II) If $\sigma-(1+\e)\gamma$ is not  $\kappa$-almost decreasing, then $J\not \in \DD(U)$.
\end{theorem}

\ms\no Let us point out that even finding an analog for the above statement for $\kappa<0$ presents an open problem. Such MIFs appear in some of the applications
mentioned in \cite{MIF2}.

\ms\no To finish this section let us reformulate the last theorem using the relations of TO.

\begin{theorem}
	In the conditions of Theorem \ref{tMIF2},
	
\ms

\ms\no I) If $\sigma-(1-\e)\gamma$ is $\kappa$-almost decreasing, then $J\tl U$;

\ms\no II) If $J\tleq U$ then $\sigma-(1+\e)\gamma$ is  $\kappa$-almost decreasing.
\end{theorem}

\begin{proof}
	Once again, The general case can be reduced to the case $J=B_\L$: otherwise
	replace the singular factor of $J$ with its Frostman transform ${\bf b}_w(S^a)$.

\ms\no 	I) If $\sigma-(1-\e)\gamma$ is $\kappa$-almost decreasing,
	then by Theorem \ref{tMIF2} there exists a non-trivial function $f\in N[\bar U JB]$, where $B$ is any finite Blaschke product. Denote the zeros
	of $B$ by $a_1,...a_n$.
	Note that then
	$$h=\frac f{(z-a_1)(z-a_2)...(z-a_n)}\in N[\bar U J].$$
	If $n=n(\kappa)$ is large enough, $h$ is bounded because $|f|<C|U'|^{1/2}$.
	Suppose now that $I\in\DD(J)$, i.e., there exists non-trivial $g\in N[\bar J I]$. Then
		$$\bar U I hg=(\bar U J h )(\bar J I  g)\in \bar H^2,$$
	which implies that $I\in \DD(U)$. Hence $\DD(J)\subset\DD(U)$ and $J\tl U$.

\ms\no II) If	$\sigma-(1+\e)\gamma$ is not $\kappa$-almost decreasing then $\sigma^*-(1+\e)\gamma$ is not $\kappa$-almost decreasing where
$\sigma^*=\arg J/b_a$ for some zero $a$ of $J$. By Theorem \ref{tMIF2} it means that $J/b_a\not\in \DD(U)$, while $J/b_a\in\DD(J)$.
Hence the relation $J\tleq U$ does not hold.

	\end{proof}

\subsection{The Type problem}
Like the Beurling-Malliavin problem, the Type problem concerns  completeness of complex exponentials in $L^2$-spaces.
This time one considers $L^2(\mu)$ for a general finite positive measure $\mu$ on $\R$ and studies completeness of families
of exponential functions with frequencies from a fixed interval. We define the type of  $\mu$ as
$$\TT_\mu=\inf \{ a | e^{ist}, s\in [-a,a],\text{ are complete in }\ L^2(\mu)\}.$$
The problem is to find $\TT_\mu$ in terms of $\mu$. This problem was considered by N. Wiener (in an equivalent form, \cite{WienerBook}) A. Kolmogorov and M. Krein, see \cite{Krein1, Krein2} or
\cite{CBMS} for further discussion and references. Using the Toeplitz approach, a formula for $\TT_\mu$ was recently found in
\cite{Type}, see also \cite{CBMS}. The idea of the Toeplitz approach to the Type problem can be expressed in terms of Toeplitz order in the following form.
Recall that for a positive singular Poisson-finite measure $\mu$  we denote by $\theta_\mu$ the inner function with Clark measure $\mu$.
The general case of the Type problem can be easily reduced to the singular case. 

\begin{theorem}\label{prop600} Let $\mu$ be a positive singular Poisson-finite measure. Then
	$$\TT_\mu=\sup \{a |\ S^a\in \DD(\theta_\mu)\}.$$
\end{theorem}

\ms\no We say that an inner function $\theta$ divides a Cauchy integral $K\mu$ for some finite complex measure $\mu$ if $K\mu/\theta\in H^p$ for some $p>0$.
Note that then $K\mu/\theta=K\nu$ where $\nu$ is another finite complex measure, $\nu=\bar{\theta}\mu$ \cite{PoltClark}.

\begin{proof}[Proof of Theorem \ref{prop600}]
	Recall that according to the Clark formula every function from $K_\theta, \ \theta=\theta_\mu$ can be represented in the form
	$f=(1-\theta)Kf\mu$. Since $1-\theta$ is an outer function in $\C_+$,
	$$\sup \{a |S^a\in \DD(\theta_\mu)\}=\sup \{a |S^a\textrm{ divides }f\in K_\theta\}=
	$$$$=\sup \{a |S^a\textrm{ divides }Kf\mu,\ f\in L^2(\mu)\}.$$
	By a theorem of Aleksandrov \cite{Alexandrov} $S^a$ divides $Kf\mu$ iff $f\perp e^{ist}, s\in [-a,a]$. Such an $f$ exists iff the family of exponentials
	$e^{ist}, s\in [-a,a]$ is incomplete in $L^2(\mu)$.
	
	\end{proof}

Utilizing the Beurling-Malliavin multiplier theorem one can deduce the following statement.

\begin{theorem}\label{prop601} Let $\mu$ be a positive singular Poisson-finite measure. Then
	$$\TT_\mu=\sup \{a |\ S^a\tleq\theta_\mu\}.$$
\end{theorem}

\subsection{Sampling measures}\label{secS}

\ms\no  Let $\mu,\nu$ be two positive Poisson-finite measures such that the Hilbert (Cauchy) transform is bounded from $L^2(\mu)$
to $L^2(\nu)$. Initially one can understand this property in the sense that for a dense family of functions $f\in L^2(\mu)$
the Cauchy integral $Kf\mu$ in the upper half-plane has non-tangential boundary values $f^*(x)$ at $\nu$-a.e. point $x$ and the
norm estimate $||f^*||_{L^2(\nu)}\leq C||f||_{L^2(\mu)}$ holds for all $f$ from that family with a uniform $C$. It follows from a theorem by Aleksandrov \cite{Alexandrov} that
then $f^*$ actually exists $\nu$-a.e. for all $f\in L^2(\nu)$ (and the same norm estimate holds). The general two-weight Hilbert problem asks to describe
pairs of measures with this property.

\ms\no Extensive studies of the 'Tauberian' version of the two-weight Hilbert problem were started in \cite{NTV} and  recently completed in \cite{L, L&Co}. These important results produced
a real analytic description of pairs $\mu$ and $\nu$. Our goal in this section is to connect this  problem with TO.

\ms\no Once again, if $\mu$ is a positive singular Poisson-finite measure on $\hat \R$ we denote by $\theta_\mu$ the corresponding inner function, i.e., the function
whose Clark measure is $\mu$.
By a theorem from \cite{PoltClark}, every function $f$ from the model space $ K_{\theta_\mu}$ has non-tangential boundary values a.e. with respect to $\mu$.
The operator of embedding $ K_{\theta_\mu}\to L^2(\mu)$ is a unitary operator. As was mentioned before, this statement generalizes the Parseval theorem from $K_S$ and the counting measure of $\Z$, which is the Clark measure for $S$, to an arbitrary model space and the corresponding Clark measure. The function $f\in K_{\theta}$ can be recovered from its boundary values in $L^2(\mu)$ via the formula $f=(1-\theta)Kf\mu$.

\ms\no Some of these connections have already been used in our discussion of TO. To summarize these relations let us recall that the dominance set
of $\theta=\theta_\mu$ is the set of all inner divisors of functions from $K_\theta$. As was discussed in the last section,
	$$\DD(\theta_\mu)=\{I|I \textrm{ is an inner divisor of }Kf\mu,\ f\in L^2(\mu)\}.$$

\ms\no Let us now return to a pair of Poisson-finite measures $\mu$ and $\nu$ such that the Cauchy transform is bounded from $L^2(\mu)$
to $L^2(\nu)$. In view of the above, this is equivalent to saying that  $K_{\theta_\mu}$ is embedded (via passing from a function
to its non-tangential boundary values) into $L^2(\eta), \ \eta=|K\mu|^2\nu$ (or $|1-\theta_\mu|^{-2}\nu$). Note that under the
condition of boundedness of the Cauchy transform, the integral $K\mu$, or equivalently the inner function $\theta_\mu$, have non-tangential boundary
values $\nu$-a.e. and the above definition of $\eta$ makes sense.

\ms\no In the case when the measures $\mu$ and $\nu$ are discrete the condition of boundedness of the Cauchy transform can be reformulated in terms of
de Branges spaces. Recall that we denote by $E_\mu$ an Hermite-Biehler function such that $E_\mu K_{\theta_\mu}=B(E_\mu)$. The boundedness of the Cauchy transform is equivalent to the boundedness of the natural embedding of  $B(E_\mu)$  into $L^2(\gamma), \ \gamma=|E_\mu K_\mu|^2\nu$. Note
that if $E_\mu=A_\mu+iB_\mu$ is the standard representation of $E_\mu$ ($A_\mu,B_\mu$ are real entire functions, $2A_\mu=E_\mu+E_\mu^\#,\ 2iB_\mu
=E_\mu-E^\#_\mu$) then $|E_\mu K_\mu|=(A_\mu^2+B_\mu^2)/|B_\mu|$.

\ms\no We say that a positive measure $\nu$ on $\hat \R$ is sampling for a Banach space $H$ of analytic functions in $\C_+$ if the non-tangential limits
$f^*(x)$ exist $\nu$-a.e. for a dense family of $f\in H$ and
$$||f||_H\asymp ||f^*||_{L^2(\nu)}.$$

\ms\no An important case of the two-weight Hilbert problem is when
$$||f||_{L^2(\mu)}\asymp ||Kf\mu||_{L^2(\nu)}.$$
In view of our discussion above, this is equivalent to the property that $\eta=|K\mu|^2\nu$ is a sampling measure for
$K_{\theta_\mu}$. The general property, when the Cauchy transform is only norm-bounded from above, can be reduced to the sampling case by adding
the Clark measure $\mu_{-1}$ to $\eta$.
Namely, if $\mu=\sigma_1$ is the Clark measure for $\theta$, let us denote by $\mu_{-1}=\sigma_{-1}$ the Clark dual measure, see Section \ref{secMod}.
The Cauchy transform is bounded from $L^2(\mu)$
to $L^2(\nu)$ iff $\tau=\eta+\mu_{-1}$ is a sampling measure for $K_{\theta}$.

\ms\no Reformulating Clark theory for MIFs in terms the corresponding de Branges spaces, we may notice that for any Poisson-finite positive discrete measure $\mu$ on $\R$
there exists a unique regular de Branges space $B(E)$ such that
$B(E)=L^2(\mu)$ and $\supp\mu=\{E=\bar E\}$. We will denote the corresponding HB function
by $E^\mu$ and the MIF $(E^\mu)^\#/E^\mu$ by $I^\mu$. The measure $\mu$ is called a de Branges measure for $B(E^\mu)$. Note the following clear
connection with the Clark measure $\sigma$ for $I^\mu$:
$$\mu=\sigma/|E^\mu|^2.$$
Other Clark measures $\sigma_\alpha,\ \alpha\in\T$ produce other de Branges measures to form the family of de Branges measures for the given space.

\ms\no As we saw above, the two-weight Hilbert problem is directly related to the problem of description of sampling measures for model spaces $K_\theta$.
If $\theta$ is a MIF and $\nu$ is a discrete Poisson-finite measure on $\R$ then $\nu$ is sampling for $K_\theta$ if and only if $\nu/|E_\theta|$
is sampling for $B(E_\theta)$, where $E_\theta$ is any HB function such that $(E_\theta)^\#/E_\theta=\theta$. Thus, in the case of discrete measures,
the two-weight problem connects to the description of sampling measures for de Branges spaces.

\ms\no Finally, for the last problem we have the following reformulation in terms of TO.
Any measure satisfying
$$||f||_{B(E)}=||f||_{L^2(\mu)}$$
is called a spectral measure for $B(E)$. Any de Branges space $B(E)$ possesses an infinite family of spectral measures with the de Branges measure defined above being one of them. The spectral measures for a given de Branges space are de Branges measures for the space, de Branges measures for larger de Branges
spaces in the chain which contains the given space, and limits of such measures along the chain.
The set of spectral measures of a given de Branges space is quite well understood in Krein-de Branges theory.
Those measures are spectral measures for the corresponding Krein canonical systems of differential equations, see \cite{dBr, DM, MPS}.

\ms\no The following statement follows from Theorem \ref{t500}.

\begin{theorem}
Let $\mu,\nu$ be two positive discrete Poisson-finite measures on $\R$. TFAE

\ms\no 1) The Hilbert (Cauchy) transform is bounded from $L^2(\mu)$
to $L^2(\nu)$.

\ms\no 2) The measure $\eta=(\nu+\mu_{-1})/|E_\mu|^2$ is a spectral measure for some $B(F)$ such that  $\theta_F\tsim \theta_\mu$.
\end{theorem}

\ms\no Note that the condition $\theta_F\tsim \theta_\mu$ means that the inner factors of functions from $B(F)$ in $\C_+$ are the same as
inner factors of Cauchy integrals $Kf\mu,\ f\in L^2(\mu)$.

\ms\no In \cite{OS} a theorem by de Branges from \cite{dBr} was applied to describe sampling sequences for the Paley-Wiener space.
Recall that the Paley-Wiener space is a de Branges space with $E=S^{-1}$.
Using the same ideas
we can formulate the following statement in terms of TO.

\begin{theorem} $\nu$ is a sampling measure for
$B(E)$ iff
$$P\nu=\Re\frac{F+F^\#\phi}{F-F^\#\phi}$$
for some HB function  $F$  such that $\theta_F\tsim \theta_E$ and some $\phi\in H^\infty,\ ||\phi||\leq 1$.
\end{theorem}

\section{Concluding remarks and further questions}

\ms\no For  decades numerous problems in complex function theory were motivated by Functional and Spectral Analysis.
This tradition can be traced to Beurling's description of invariant subspaces of the shift operator using
inner functions, proofs of  uniqueness theorems for Schr\"odinger operators by Borg and Marchenko, etc.
One of the main sources
for such problems is the Nagy-Foias functional model theory mentioned in this text, see \cite{Nikolski}. Problems
appearing in this context are related to inner functions and model spaces $K_\theta$. The Krein-de Branges theory \cite{dBr, DM}, created
to treat spectral problems for differential operators, is another source of such problems where the main objects are
entire functions and meromorphic inner functions.

\ms\no Recent developments in the Toeplitz approach to UP have raised a large number of new  questions. One of the main goals of this paper was to outline some of such questions and
bring them to the attention of the experts in complex function theory. In conclusion, let us give a brief summary of problems
appearing in relation to TO. This is only a small sample of questions from the area of the Toeplitz approach, and interested reader
will be able to find many more challenging problems in other sources, including those cited in this note.

\ms\no Giving metric conditions for TO and TE seems to be a natural question.
In particular, producing conditions for two Blaschke products $B_1$ and $B_2$ to satisfy $B_1\tsim B_2$ or $B_1\tleq B_2$ in terms
of their zeros or arguments presents the main version of this problem. While giving 'if and only if' conditions for general inner functions
may be out of reach at the moment, solving this problem for the inner functions from a restricted class like in Theorem \ref{tMIF2} or fixing one of the functions like in Theorems \ref{tBMTO1} and \ref{tBMTO2} would still be interesting (and challenging). Let us point out that one of the important cases in Theorem \ref{tMIF2} when $\kappa<0$, which appears in applications (see \cite{MIF2}), is still open. Such results would provide generalizations
of BM theory and numerous other applications.

\ms\no As we saw in the text, study of the structure of the dominance set $\DD(\theta)$, and in particular providing metric conditions in terms of $\theta$ for
a given function $I$ to be a base or total element in $\DD(\theta)$, generalizes problems on completeness, minimality and zero sets in model and de Branges spaces. While such problems are well understood in the standard Paley-Wiener spaces, where they become equivalent to BM problem and related problems on
exponential families, as well as spectral problems for regular Schr\"odinger and Dirac operators (see \cite{MIF1}), they are mostly open for $\theta\not = S$.

\ms\no Among questions more specific to Toeplitz order, let us mention the natural question of finding supremum or infimum, with respect to the relations
$\tleq$ and $\tgeq$, for a given collection of inner functions. Starting with finite collections and functions from restricted classes, this
problem can reach any desired level of difficulty. Even its first step, determining the existence of an upper (lower) bound, which is simple
for finite collections, can be interesting for the infinite ones.

\ms\no Another problem mentioned in the text is the problem of finding the Toeplitz hull for a given collection of functions,
i.e., the smallest dominance set $\DD(\theta)$ containing the collection. Some of the cases of this problem are equivalent to
finding the smallest de Branges or model space for a given collection of zero sets.  A version of the same question corresponds
to finding the differential operator for which the given sequences are (are not)  defining, in
the terminology of \cite{MIF2}.

\ms\no Apart from pure function theoretic problems, finding further connections from the objects of this note to problems of UP and spectral theory remains one of natural directions for research.


\begin{thebibliography}{24}


\bibitem{Alexandrov} {\sc Aleksandrov, A.} {\it Isometric embeddings of
	coinvariant subspaces of the shift
	operator,}  Zap. Nauchn.
Sem. S.-Peterburg. Otdel. Mat.
Inst. Steklov.
(POMI)  232
(1996),  Issled. po Linein Oper. i Teor. Funktsii. 24,
5--15,
213; translation in  J. Math. Sci. (New York)  92
(1998),
no.~1, 3543--3549

\bibitem{ACL} {\sc E. Andruchow, E. Chiumiento and  G. Larotonda} {\it Geometric significance of Toeplitz kernels,}  preprint.


\bibitem{BS}{\sc Baranov, A. and Sarason, D.} {\it Quotient representations of inner functions,} Recent Trends in Analysis, Proceedings of the conference in honor of Nikolai
Nikolski, Theta Foundation, Bucharest, 2013, pp.
35-46.


\bibitem{Belov}{\sc Belov, Yu.} {\it Complementability of exponential systems,} Comptes Rendus Math.,
Vol. 353, Issue 3, March 2015,  215-218

\bibitem{BM1}  {\sc Beurling, A., Malliavin, P.} {\it
	On Fourier transforms of measures with compact support,}
Acta Math.  107 (1962), 291--302

\bibitem{BM2}  {\sc Beurling, A., Malliavin, P.} {\it
	On the closure of characters and the zeros of entire functions,}
Acta Math.  118 (1967), 79-93



\bibitem{dBr}  {\sc De Branges, L.} {\it Hilbert spaces of entire functions.} Prentice-Hall,
Englewood Cliffs, NJ, 1968
	
	
\bibitem{Cl} {\sc Clark, D.}  {\it One dimensional perturbations of restricted shifts,}
J.  Anal.  Math. 25 (1972),   169-91.


\bibitem{Dya}{\sc Dyakonov, K.} {\it Zero sets and multiplier theorems for star-invariant subspaces.} J. Anal. Math. 86 (2002),
247-269.


\bibitem{DM}{\sc Dym H, McKean H.P.} {\it Gaussian processes, function theory and the inverse spectral problem}
Academic Press, New York,  1976


\bibitem{FHR}{\sc E. Fricain, A. Hartmann and B. Ross,} {\it  Multipliers between model spaces.} Studia Math., 240(2), 177–191, 2018.



\bibitem{Garnett}{\sc Garnett, J.  } {\it Bounded analytic functions.} Academic Press, New York, 1981



\bibitem{HJ}{\sc Havin, V. P.,  J\"oricke, B.} {\it The uncertainty principle in harmonic analysis.}
Springer-Verlag, Berlin, 1994.


\bibitem{HNP}{\sc Hruschev S.,  Nikolskii, N.,  Pavlov,  B.} {\it
	Unconditional bases of
	exponentials and of reproducing kernels,}
Lecture Notes in Math., Vol.  864,  214--335

\bibitem{Koosis} {\sc  Koosis, P.} {\it The logarithmic integral, Vol. I \ \& II,} Cambridge Univ. Press, Cambridge, 1988

\bibitem{KoosisHp} {\sc  Koosis, P.} {\it Introduction to $H^p$ spaces.} Cambridge Univ. Press, Cambridge, 1980

\bibitem{Krein1} {\sc Krein, M. G.} {\it On an extrapolation problem of A. N. Kolmogorov,} Dokl. Akad. Nauk SSSR 46 (1945),
306--309 (Russian).


\bibitem{Krein2} {\sc Krein, M. G.} {\it On a basic approximation problem of the theory of extrapolation and filtration of stationary random processes,} Doklady Akad. Nauk SSSR (N.S.)  94,  (1954),  13--16 (Russian).

\bibitem{L}{\sc Lacey, M. T.}{\it The Two Weight Inequality for the Hilbert Transform: A Primer}, preprint
	
\bibitem{L&Co}{\sc Michael T. Lacey, Eric T. Sawyer, Chun-Yen Shen, Ignacio Uriarte-Tuero} {\it Two Weight Inequality for the Hilbert Transform: A Real Variable Characterization, I}
Duke Math. J. 163, no. 15 (2014), 2795-2820	
	
	
\bibitem{MIF1} {\sc  Makarov, N.,  Poltoratski, A.} {\it Meromorphic inner functions, Toeplitz kernels, and the uncertainty principle,} in {\it Perspectives in Analysis}, Springer Verlag, Berlin, 2005, 185--252

\bibitem{MIF2} {\sc  Makarov, N.,  Poltoratski, A.} {\it  Beurling-Malliavin theory for Toeplitz kernels,}
Invent. Math., Vol. 180, Issue 3 (2010), 443-480
	
\bibitem{Uncertainty}{\sc Makarov, N., Poltoratski, A.} {\it Two Spectra Theorem with Uncertainty,}	 preprint

\bibitem{Etudes}{\sc Makarov, N., Poltoratski, A.} {\it Etudes in spectral problems,}	 preprint

	
	
\bibitem{MPS}{\sc Makarov, N., Poltoratski, A., Sodin, M.} {\it Lectures on Linear Complex Analysis}, in preparation.
	
	
	
\bibitem{Nikolski} {\sc Nikolskii, N. K.} {\it Treatise on the shift
	operator},  Springer-Verlaag, Berlin (1986)
	
\bibitem{NTV}{\sc F. Nazarov, S. Treil and A. Volberg,} {\it Two weight estimate for the Hilbert transform and corona decomposition for non-doubling measures,} preprint.


	
\bibitem{OS}  {\sc  Ortega-Cerd\`a, J., Seip, K.} {\it Fourier frames,} Annals of Math. 155 (2002), 789--806
	


\bibitem{PoltClark}{\sc  A. Poltoratski}, {\it On the boundary behavior of pseudocontinuable
	functions,}
St. Petersburg Math. J.,{\bf 5 } (1994), 389--406.



\bibitem{PolSar} {\sc Poltoratski, A. and Sarason, D.} {\it Aleksandrov-Clark measures,} Recent advances in operator-related function theory, 1--14, Contemp. Math., 393, Amer. Math. Soc., Providence, RI, 2006


\bibitem{GAP} {\sc Poltoratski, A.} {\it Spectral gaps for sets and measures}, Acta Math., 2012, Volume 208, Number 1, pp. 151-209.



\bibitem{Type}{\sc
	Poltoratski, A.},
{\it Problem on completeness of exponentials},
Ann. Math.,
178,
983--1016,
2013

\bibitem{CBMS}{\sc  Poltoratski, A.} {\it Toeplitz Approach to Problems of the Uncertatinty Principle,}	book in CBMS seies, AMS/NSF, 2015


\bibitem{SY} {\sc Sodin, M. and  Yuditskii, P.} {\it Another approach to de Branges� theorem on weighted polynomial
	approximation,} Proceedings of the Ashkelon Workshop on Complex Function Theory
(1996), 221�227, Israel Math. Conf. Proc., 11, Bar-Ilan Univ., Ramat Gan, 1997.

\bibitem{WienerBook} {\sc Wiener, N.} {\it Extrapolation, Interpolation, and Smoothing of Stationary Time Series,} M.I.T. press, 1949.

\end{thebibliography}
\end{document}